\documentclass[amscd,amssymb,verbatim,12pt]{amsart}
\numberwithin{equation}{section}
\usepackage{amssymb}
\usepackage{epsfig}
\usepackage{multido}
\usepackage{color}
\usepackage{pstricks}
\usepackage{pst-all}
\usepackage{pst-math}
\usepackage{pst-func}
\usepackage{pst-3dplot}

\newtheorem{thm}{Theorem}[section]
\newtheorem{cor}[thm]{Corollary}
\newtheorem{lem}[thm]{Lemma}

\theoremstyle{definition}

\newtheorem{rem}[thm]{Remark}

\newif\ifShowLabels
\ShowLabelstrue
\newdimen\theight
\def\TeXref#1{
     \leavevmode\vadjust{\setbox0=\hbox{{\tt
            \quad\quad  {\small  \bf #1}}}%
     \theight=\ht0
     \advance\theight  by  \dp0
     \advance\theight  by  \lineskip
     \kern -\theight \vbox  to
     \theight{\rightline{\rlap{\box0}}%
      \vss}%
      }}%

    {\begin{thm}\label{#1} \ifShowLabels \TeXref{#1} \fi}%
    {\end{thm}}

    {\begin{def}\label{#1} \ifShowLabels \TeXref{#1} \fi}%
    {\end{def}}

    {\begin{lem}\label{#1} \ifShowLabels \TeXref{#1} \fi}%
    {\end{lem}}

    {\begin{cor}\label{#1} \ifShowLabels \TeXref{#1} \fi}%
    {\end{cor}}

\newcommand{\eqRef}[1]%
     {\ifShowLabels \TeXref{#1} \fi
      \begin{equation}\label{#1} }

\ShowLabelsfalse
\newcommand{\vsp}{\vskip 1em}
\newcommand{\vspp}{\vskip 2em}
\newcommand{\NI}{\noindent}
\newcommand{\bea}{\begin{eqnarray}}
\newcommand{\eea}{\end{eqnarray}}
\newcommand{\IR}{I\!\!R}
\newcommand{\be}{\begin{equation}}
\newcommand{\ee}{\end{equation}}
\newcommand{\ben}{\begin{eqnarray*}}
\newcommand{\een}{\end{eqnarray*}}
\newcommand{\lm}{\lambda}
\newcommand{\Om}{\Omega}
\newcommand{\om}{\omega}
\newcommand{\p}{\partial}
\newcommand{\al}{\alpha}

\newcommand{\Lm}{\Lambda}

\newcommand{\ve}{\varepsilon}
\newcommand{\dl}{\delta}

\newcommand{\etb}{\bar{\eta}}
\newcommand{\tht}{\theta}

\newcommand{\hh}{\hat{h}}

\newcommand{\hk}{\hat{k}}

\newcommand{\hT}{\hat{T}}

\newcommand{\hs}{\hat{s}}
\newcommand{\hve}{\hat{\epsilon}}

\newcommand{\hal}{\hat{\alpha}}

\newcommand{\lam}{\lambda}
\newcommand{\gm}{\gamma}
\newcommand{\G}{\Gamma}

\newcommand{\s}{\sigma}

\newcommand{\Dp}{\Delta_p}
\newcommand{\XO}{ \mathcal{P}^+_{\Om_T} }

\newcommand{\XC}{ \overline{\mathcal{P}}^+_{\Om_T} }
\newcommand{\YC}{ \overline{\mathcal{P}}^-_{\Om_T} }

\title[Trudinger's equation]{On the viscosity solutions to Trudinger's equation}

\author[Bhattacharya, Marazzi]{Tilak Bhattacharya and Leonardo Marazzi}

\address{Department of Mathematics, Western Kentucky University, Bowling Green, Ky 42101\\
\vsp
Department of Liberal Arts, Savannah College of Arts and Design, Savannah, Ga 31401}
\thanks{AMS Math Subject Classification 2010: 35K65, 35K55}
\begin{document}

\maketitle

\begin{abstract} We study the existence of positive viscosity solutions to Trudinger's equation for cylindrical domains $\Om\times[0, T)$, where $\Om\subset \IR^n,\;n\ge 2,$ is a bounded domain, $T>0$ and $2\le p<\infty$. We show existence for general domains $\Om,$ when $n<p<\infty$. For $2\le p\le n$, we prove existence for domains $\Om$ that satisfy a uniform outer ball condition. We achieve this by constructing suitable sub-solutions and super-solutions  and applying Perron's method.
\end{abstract}
\section{Introduction}
\vsp
\NI In this work, we study the existence of positive viscosity solutions to Trudinger's equation. This is a follow-up of the work in \cite{BL} where we studied a doubly nonlinear parabolic differential equation involving the infinity-Laplacian. Our goal in the current work is to adapt and apply the ideas developed in \cite{BL} to show existence of solutions to an analoguous
equation involving the $p$-laplacian. 
\vsp
\NI To make our discussion more precise, we introduce some definitions and notations. Let $\Om\subset \IR^n,\;n\ge 2,$ be a bounded domain and $T>0$. Define 
$\Om_T=\Om\times(0,T)$ and $P_T=(\Om\times \{0\})\cup (\p\Om\times [0,T))$ its parabolic boundary. 
\vsp
\NI Let $f\in C(\Om,\IR^+),\;g\in C(\p\Om\times [0,T),\IR^+)$ and $u:\Om_T\times P_T\rightarrow \IR^+$. We take $2\le p<\infty$. Our goal here is to show the existence of a
positive viscosity solution $u$, continuous on $\Om_T\cup P_T$, to the equation  
\bea\label{sec1.1}
&&\mbox{div}(|Du|^{p-2}Du)-(p-1)u^{p-2} u_t=0 \;\;\;\mbox{in $\Om_T$ },\\
&&u(x,0)=f(x)\;\;\mbox{and}\;\;u(x,t)=g(x,t),\;\forall\;(x,t)\in \p\Om\times [0,T).\nonumber
\eea
This doubly nonlinear parabolic differential equation is referred to as Trudinger's equation, see \cite{ED, GV, KK, TR} and the references therein  where results have been stated in the context of weak solutions. Our effort is to discuss existence in the context of viscosity solutions, see \cite{CI, CIL}. The operator
$$\Dp u=\mbox{div}(|Du|^{p-2}Du),\;\;\;2\le p<\infty,$$
is referred to as the $p$-Laplacian and is degenerate elliptic. Parabolic equations involving the $p$-Laplacian have been extensively studied and a detailed discussion may be found in \cite{ED}. The equation in (\ref{sec1.1}) is doubly nonlinear and such equations are also of great interest. 
\vsp
\NI In our study of (\ref{sec1.1}) a central role is played by the following parabolic equation:
\eqRef{sec1.2}
\Dp \eta+(p-1)|D\eta|^p-(p-1)\eta_t=0,\;\mbox{in $\Om_T$},\;\;\;2\le p<\infty,
\ee
with $\eta=\hat{f},$ on $\Om\times\{0\},$ and $\eta=\hat{g},$ on $\p\Om\times[0,T)$, where both $\hat{f}$ and $\hat{g}$ are continuous.
As a matter of fact, we will show that if $u$ solves (\ref{sec1.1}) and $\eta=\log u$ then (\ref{sec1.2}) holds. A great part of this work employs this equivalence to prove the existence of $u$. One may find a detailed study of this equation in \cite{ED, OP, PM}, for instance. Incidentally, the related equation
$\Dp u-u_t=0$ has also been studied in the viscosity setting, see for instance \cite{JL, LM} and references therein. A discussion in the weak solution setting may be found in \cite{ED}. Some of our results do hold for this equation, however, our focus is primarily the study of (\ref{sec1.1}) and (\ref{sec1.2}). 
\vsp
\NI We state our main results as two theorems. The first result addresses the case $n<p<\infty$ and we show existence for general domains.
The second result states the existence result for $2\le p\le n$ and holds for domains $\Om$ that satisfy a uniform outer ball condition. At this time, it is not clear to us how to extend the result to general domains. Our proofs of existence make use of the Perron method, see \cite{CI, CIL} for more details. 
In order to do so we prove a comparison principle for (\ref{sec1.1}) and (\ref{sec1.2}). Incidentally, the comparison principle for (\ref{sec1.2}) implies a quotient type comparison for positive solutions of (\ref{sec1.1}). The major part of this work is devoted to the construction of suitable sub-solutions and super-solutions for (\ref{sec1.1}). More precisely, at each point on $P_T$, we construct sub-solutions and super-solutions which are arbitrarily close to the given data at the point. Using the parabolic counterpart of Theorem 4.1 in \cite{CIL} (see \cite{CI}), we conclude existence. For a short description, see Section 5 in \cite{BL}. 
To achieve our goal the equation in (\ref{sec1.2}) proves particularly useful when working with side conditions along $\p\Om\times [0,T)$. 
\vsp
\NI For ease of presentation, we set
\eqRef{sec1.3}
h(x,t)=\left\{ \begin{array}{lcr} f(x),&\;\forall x\in\Om,\;t=0,\\ g(x,t),&\;\forall(x,t)\in \p\Om\times[0,T). \end{array}\right. 
\ee
\vsp
\NI We obtain
\begin{thm}\label{sec1.4}
Let $\Om\subset \IR^n,\;n\ge 2$, be a bounded domain and $T>0$. Let $h\in C(P_T, \IR^+)$, with $\inf_{P_T} h>0$, and $n<p<\infty$. The problem
\ben 
\Dp u=(p-1) u^{p-2} u_t,\;\;\mbox{and \quad$u(x,t)=h(x,t),\;\forall(x,t)\in P_T$}.
\een
has a unique positive viscosity solution $u$ that is continuous on $\Om_T\cup P_T$.
\end{thm}
\begin{thm}\label{sec1.5}
Let $\Om\subset \IR^n,\;n\ge 2$, be a bounded domain and $T>0$. Let $h\in C(P_T, \IR^+)$, with $\inf_{P_T}h>0$, and $2\le p\le n$. If $\Om$ satisfies an uniform outer ball condition then
the problem
\ben 
\Dp u=(p-1) u^{p-2} u_t,\;\;\mbox{and \quad$u(x,t)=h(x,t),\;\forall(x,t)\in P_T$}.
\een
has a unique positive viscosity solution $u$ that is continuous on $\Om_T\cup P_T$.
\end{thm}
\NI As corollaries, our work implies existence of solutions of (\ref{sec1.2}) for any bounded continuous initial and boundary data. 
\vsp
\NI We now describe the general layout of the work. In Section 2, we provide definitions and a change of variables formula showing the equivalence of (\ref{sec1.1}) and (\ref{sec1.2}). Some additional change of variables results are also stated. These will be followed by a result on a separation of variables. Section 3 contains a maximum principle, a comparison principle and their consequences. The maximum principle requires no sign conditions, however the main comparison principle holds only for positive solutions. This leads to a quotient version of the comparison for solutions to (\ref{sec1.1}), see \cite{BL}. 
In Section 4, we construct sub-solutions and super-solutions for the initial data ($t=0$) for $2\le p<\infty$. is done in Section 4. For this purpose, we work directly with  the parabolic equation in (\ref{sec1.1}). This work is valid for general domains. We discuss the side conditions in Sections 5 and 6 and make use of the parabolic equation in (\ref{sec1.2}).
We address the case $n<p<\infty$ for general domains in Section 5. Section 6 takes up the case $2\le p\le n$ for domains that satisfy a uniform outer ball condition. 
\vsp
\section{Notations, definitions and some preliminaries}
\vsp
\NI In what follows, $\Om\subset \IR^n,\;n\ge 2,$ is a bounded domain and $\p\Om$ its boundary. Let $\overline{A}$ denote the closure of a set $A$ and
$A^c$ the complement of $A$ in $\IR^n$. The letters $x,\;y,\;z$ are used for points in $\IR^n$ and we reserve the letter $o$ for the origin in $\IR^n$. 
The letters $s,\; t$ denote points in $\IR^+\cup\{0\}$.
We use the standard notation for $x\in\IR^n,$ i.e. $x=(x_1,x_2,\cdots, x_n)$. 
For $T>0$, we define the cylinder
\eqRef{sec2.1}
\Om_T=\Om\times (0,T)=\{(x,t)\in \Om\times \IR:\;x\in \Om,\;0<t<T\}.
\ee
The parabolic boundary $P_T$ of $\Om_T$ is the set $(\Om\times \{0\}) \cup (\p \Om\times [0,T) ).$
\vsp
\NI Let $B_r(x)\subset \IR^n$ be the open ball of radius $r$, centered at $x$. For $r>0$ and $\tau>0$, we define the following open cylinder
\eqRef{sec2.2}
D_{r,2\tau}(x,t)=B_r(x)\times \left(t-\tau, t+\tau \right).
\ee  
\vsp
\NI In this work we will always take $2\le p<\infty$. We now define the operators $\Pi$ and $\G$ as follows.
\eqRef{sec2.3}
\Pi(\phi)=\Dp \phi-(p-1) \phi^{p-2} \phi_t\;\;\;\mbox{and}\;\;\;\G(\eta)=\Dp \eta+(p-1)|D\eta|^p-(p-1) \eta_t.
\ee
\vsp
\NI For studying viscosity solutions of (\ref{sec1.1}) and (\ref{sec1.2}), we define expressions which are useful in this context and will apply to the operators in (\ref{sec2.3}). Let 
$\mathcal{S}(n)$ be the set of symmetric $n\times n$ matrices and $tr(X)$ denote the trace of a matrix $X\in S(n)$. For any $(r,a,q, X)\in \IR\times\IR\times \IR^n\times S(n)$, we define the expressions $L_p(q, X)$ and $H_p(r,a,q, X)$ as follows.
Define, for $2\le p<\infty$,
\bea\label{sec2.4}
L_p(q, X)=\left\{ \begin{array}{cc} |q|^{p-2}tr(X)+(p-2)|q|^{p-4} q_i q_j X_{ij},& \;\mbox{for $q\ne 0$},\\
                             tr(X),& \;\mbox{$p=2$},\\
                          0,& \;\mbox{for $q=0,\;\;p\ne 2.$} \end{array}\right.
\eea
The quantity $L_p(q, X)$ is the version for the differentiated $p$-Laplacian and $L_2(p,X)=tr(X).$
Related to the operator $\Pi$ in (\ref{sec2.3}) is the following expression:
\eqRef{sec2.5}
 T_p(r, a,q, X)=L_p(q, X)-(p-1)|r|^{p-2} a,\;\;\forall(r,a,q,X)\in \IR\times\IR\times \IR^n\times S(n).
\ee
Similarly, related to $\G$, we have  
\eqRef{sec2.6}
K_p(a,q, X)=L_p(q,X)+(p-1)|q|^p-(p-1) a,\;\;\;\forall(a,q,X)\in \IR\times\IR^n\times S(n).
\ee
\vsp
\NI We now recall the definition of viscosity sub-solutions and super-solutions, both via semi-jets and test functions, see \cite{CIL}.
From hereon, $usc(A)$ and $lsc(A)$ denote the sets of all functions that are upper semi-continuous and lower semi-continuous on a set $A,$ respectively. 
\vsp
\NI Let $u:\Om_T\cup P_T:\rightarrow \IR$, and $(y,s)\in \Om_T$. We recall the definitions of semi-jets $\mathcal{P}^+_{\Om_T}u(y,s)$ and $\mathcal{P}^-_{\Om_T}u(y,s)$. An element $(a,q,X)\in \IR\times \IR^n\times \mathcal{S}(n)$ is in $\mathcal{P}^+_{\Om_T}u(y,s)$ if
\eqRef{sec2.7}
u(x,t)\le u(y,s)+a(t-s)+\langle p, x-y\rangle+\frac{ \langle X (x-y), x-y\rangle }{2}+o(|t-s|+|x-y|^2)
\ee
as $(x,t)\rightarrow(y,s)$, where $(x,t)\in \Om_T$. We define $\mathcal{P}^-_{\Om_T}u=-\mathcal{P}^+_{\Om_T}(-u).$ 
For the sets $\XC u(y,s)$ and $\YC u(y,s)$, see \cite{CIL}.
\vsp
\NI We present definitions of a sub-solution and a super-solution for the differential equations in (\ref{sec1.1}) and (\ref{sec1.2}). A function $u>0$ is a sub-solution in $\Om_T$ of the equation in (\ref{sec1.1}) or $\Pi(u)\ge 0$ in $\Om_T$, if $u\in usc(\Om_T)$ and 
\eqRef{sec2.8}
T_p(u(x,t), a, q, X) \ge 0,\;\;\forall(a,q,X) \in \XO u(x,t)\;\mbox{and}\;\forall(x,t)\in \Om_T.
\ee
Similarly, $v>0$ is a super-solution in $\Om_T$ or $\Pi(v)\le 0$ in $\Om_T$, if $T_p(v(x,t), a, q, X) \le 0$, $\forall(a,q,X) \in \mathcal{P}^-_{\Om_T} u(x,t)$ and $\forall(x,t)\in \Om_T.$
\vsp
\NI The definitions of a sub-solution $u$ and a super-solution $v$ of (\ref{sec1.2}) are given in terms of $K_p$, see (\ref{sec2.6}). In this case we write $\G( u)\ge 0$ and $\G (v)\le 0$ respectively. More precisely, $u$ is a sub-solution of the equation in (\ref{sec1.2}) or $\G(u)\ge 0$ in $\Om_T$ if $u\in usc(\Om_t)$ and
$$K_p(a,q, X)\ge 0,\;\;\forall(x,t)\in \Om_T\;\mbox{and}\;\forall(a,q,X) \in \XO u(x,t).$$
The definition a super-solution $v$ i.e., $\G(v)\le 0$ is similar.
\vsp
\NI We now present the definitions of a sub-solution and a super-solution in terms of test functions. In the rest of this work, we will always take a test function $\psi(x,t)$ to be $C^2$ in $x$ and $C^1$ in $t$ on $\Om_T$. 
\vsp
\NI We say that $u>0$ is a sub-solution of the equation in (\ref{sec1.1}), written as $\Pi(u)\ge 0$ in $\Om_T$, if, $u\in usc(\Om_T)$ and for all test functions $\psi$ such that $u-\psi$ has a local maximum at some $(y,s)\in \Om_T$, we have
\eqRef{sec2.9}
\Dp \psi(y,s)-(p-1)u^{p-2}(y,s)\psi_t(y,s)\ge 0.
\ee
The function $u>0$ is a super-solution i.e., $\Pi(u)\le 0$, if $u\in lsc(\Om)$ and we have that $\Dp \psi(y,s)-(p-1)u^{p-2}(y,s)\psi_t(y,s)\le 0,$ for any $\psi$, a test function, and $(y,s)\in \Om_T$ such that $u-\psi$ has a local minimum at $(y,s)$. We say $\eta$ is a sub-solution of the equation in (\ref{sec1.2}), written as $\G(\eta)\ge 0$ in $\Om_T$, if, for any test function $\psi$ and $(y,s)\in \Om_T$ such that $\eta-\psi$ has a local maximum at $(y,s)$, we have $\G(\psi)(y,s)\ge 0.$ Next, $\eta$ is a super-solution i.e., $\G(\eta)\le 0$ in $\Om_T$, if for any test function $\psi$ and $(y,s)\in \Om_T$ such that $\eta-\psi$ has a local minimum at $(y,s)$, we have $\G(\psi)(y,s)\le 0.$
\vsp
\NI We define $u$ to be a sub-solution of the problem in (\ref{sec1.1}) if $u\in usc(\overline{\Om}\times [0,T))$ and
\eqRef{sec2.10}
\Pi(u)\ge 0,\;\;\mbox{in $\Om_T$,}\;\; u(x,t)\le h(x,t),\;\forall(x,t)\in P_T,
\ee 
see (\ref{sec1.3}) for the definition of $h$. Similarly, $u$ is a super-solution of (\ref{sec1.1}) if $u\in lsc(\Om_T\cup P_T)$ and
\eqRef{sec2.11}
\Pi(u)\le 0,\;\;\mbox{in $\Om_T$,}\;\; u(x,t)\ge h(x,t),\;\forall(x,t)\in P_T
\ee
We say $u$ solves (\ref{sec1.1}) if both (\ref{sec2.10}) and (\ref{sec2.11}) hold. In this case, $u\in C(\Om_T\cup P_T)$, $\Pi(u)=0$ in $\Om_T$ and $u=h$ on $P_T$.
Similar definitions can be provided for the problem in (\ref{sec1.2}).
\vsp
\NI From hereon, all differential equations and inequalities will be understood in the viscosity sense.
\vsp
\NI We now present Lemma \ref{sec2.12} that addresses the change of variables needed for the equivalence of the equations in (\ref{sec1.1}) and (\ref{sec1.2}). This is followed by Remark \ref{sec2.14} that contains some useful observations about change of variables involving the independent variables.
\vsp
\begin{lem}\label{sec2.12}
Let $\Om\subset \IR^n,\;n\ge 2$ and $T>0$. Suppose that $\phi:\Om_T\rightarrow \IR$ and $\phi>0$. Let the operators $\Pi$ and $\G$ be as in (\ref{sec2.3}). Set $\eta=\log \phi$; then the following hold in $\Om_T$:
\ben
&(i)&\mbox{$\phi\in usc(\Om_T)$ and $\Pi(\phi)\ge 0$ if and only if $\eta\in usc(\Om_T)$ and $\G (\eta)\ge 0$. Similarly,} \\
&(ii)&\mbox{$\phi\in lsc(\Om_t)$ and $\Pi(\phi)\le 0$ in $\Om_T$, if and only if $\eta\in lsc(\Om_T)$ and $\G (\eta)\le 0$ in $\Om_T$.}
\een
\NI Thus, $\phi$ solves $\Pi(\phi)=0$ in $\Om_T$ if and only if $\eta$ solves $\G(\eta)=0$ in $\Om_T$.
\end{lem}
\begin{proof} We prove part (i), the proof of part (ii) is analogous. Suppose that $\phi>0$ solves $\Pi(\phi)\ge 0$ in the sense of viscosity. Let $\psi(x,t)$ be a test function and $(y,s)\in \Om_T$ be such that $\eta-\psi$ has a maximum at $(y,s).$ The inequality $(\eta-\psi)(x,t)\le (\eta-\psi)(y,s)$ is equivalent to
\bea\label{sec2.13}
\phi(x,t)&\le& \phi(y,s) e^{\psi(x,t)-\psi(y,s)} \eea
Set $\xi(x,t)
=\phi(y,s)\left\{ e^{\psi(x,t)-\psi(y,s)}+\psi(y,s)-1\right\}$ and 
observe that $\xi(y,s)=\phi(y,s)\psi(y,s)$ and $\xi_t(y,s)=\phi(y,s) \psi_t(y,s)$. Using \eqref{sec2.13} we see that
\eqRef{sec2.130}
(\phi-\xi)(x,t)\le\phi(y,s)[ 1-\psi(y,s)]=(\phi-\xi)(y,s).\ee
We observe
$$\left(\Dp e^{\psi} \right)(y,s)=e^{(p-1)\psi(y,s)} \left(\Dp \psi(y,s)+(p-1) |D\psi|^p(y,s)|  \right).$$
Employing the above, we obtain
\begin{multline}
\Dp \xi(y,s)
=\left(\frac{\phi(y,s)}{e^{\psi(y,s)}} \right)^{p-1} \left(\Dp e^{\psi} \right)(y,s)=(\phi(y,s))^{p-1}\left(\Dp \psi(y,s)+(p-1) |D\psi|^p(y,s)|  \right)
\end{multline}
Since $\phi$ is a sub-solution, i.e., $\Dp \xi(y,s)-(p-1)\phi^{p-2}(y,s)\xi_t(y,s)\ge 0$ (see (\ref{sec2.9}) and (\ref{sec2.130})), the definition of $\xi$
implies that
\ben
\Dp \xi(y,s)&-&(p-1)\phi^{p-2}(y,s)\xi_t(y,s)\\
&=& \phi(y,s)^{p-1} \left\{\Dp \psi(y,s)+ (p-1) |D\psi|^p(y,s)-(p-1)\psi_t(y,s) \right\}\ge 0.
\een
Thus, $\G(\psi)(y,s)\ge 0$ and the claim holds.
\vsp
\NI We prove the converse. Let $\G(\eta)\ge 0$ and suppose that $\psi(x,t)$ is a test function and and $(y,s)\in \Om_T$ is such that
$\phi-\psi$ has a maximum at $(y,s)$. Using $\eta=\log \phi$, we have
\ben
\eta(x,t)&\le& \log \left(\phi(y,s)+\psi(x,t)-\psi(y,s) \right).
\een
Set 
$$\zeta(x,t)
=\log \left(\phi(y,s)+\psi(x,t)-\psi(y,s)\right)-\frac{\psi(y,s)}{\phi(y,s)}.$$
Using the bound on $\eta$, one sees that $(\eta-\zeta)(x,t)\le (\eta-\zeta)(y,s)=\psi(y,s)/\phi(y,s)$.
Thus, $\G(\zeta)(y,s)\ge 0$, see (\ref{sec2.3}) and (\ref{sec2.9}). Differentiating,
\ben
D\zeta(y,s)=\frac{D\psi(y,s)}{\phi(y,s)},\;\zeta_t(y,s)=\frac{\psi_t(y,s)}{\phi(y,s)}\;\mbox{and}\;\Dp\zeta=\frac{\Dp \psi(y,s)}{\phi(y,s)^{p-1}}-\frac{(p-1)|D\psi|^p(y,s)}{\phi(y,s)^p}.
\een
Using these expressions in $\G(\zeta)(y,s)\ge 0$, we obtain
$$\Dp \psi(y,s)-(p-1)\phi^{p-2}(y,s)\psi_t(y,s)\ge 0.$$ 
\end{proof}

\begin{rem}\label{sec2.14} Assume that $o\in \Om$. We discuss some additional change of variables formulas. Let $\al, \;\beta,\;\lm$, $\s$ be positive constants and $\phi$ and $\eta$ be defined on $\Om$. 
Set $z=\al x,\;\om=\beta t$ and $\Om^\al_{\beta T}=\{(\al x, \beta t):\;(x,t)\in \Om\}$.
\vsp
\NI (a) Suppose that $\eta$ solves
\eqRef{sec2.1401}
\Dp \eta+ (p-1)|D\eta|^p-(p-1)\eta_t\ge (\le)0,\;\;\;\mbox{in $\Om_T$.}
\ee
Set $\varphi(z, \om)=\eta(x,t)/\lam$, differentiating $D_x\eta=\lam \al D_z\varphi$, $D_{xx}\eta=\lam \al^2D_{zz}\varphi$, $D_t \eta=\lam \beta D_\om \varphi$ and 
$\Dp \eta=\lam^{p-1}\al^p\Dp \varphi.$
\vsp
\NI (i) From (\ref{sec2.1401}), we see that
$$\Dp \varphi+\lam(p-1) |D\varphi|^p-\frac{(p-1)\beta}{\al^p\lam^{p-2}} \varphi_\om\ge (\le) 0,\;\mbox{in $\Om^\al_{\beta T}$}.$$
\NI(ii) Taking $\lam=1$ and $\beta=\al^p$ in part (i), we get
$$\Dp \varphi+(p-1) |D\varphi|^p-(p-1) \varphi_\om\ge (\le) 0,\;\mbox{in $\Om^\al_{\al^p T}$}.$$
Clearly, the Trudinger equation is invariant under this change of variables. 
Selecting $\lam=(p-1)^{-1}$ and $\beta=\al^p \lam^{p-1}$ in part (i), (\ref{sec2.1401}) yields
$\Dp \varphi+|D\varphi|^p -\varphi_\om\ge (\le )0,
\;\mbox{in $\Om^\al_{\beta T} $}.$
\vsp
\NI (iii) In (i) we choose $\beta=\lam^{p-2} \al^p$. Then (\ref{sec2.1401}) leads to
$$\Dp \varphi+\lam (p-1) |D\varphi|^p-(p-1) \varphi_\om\ge (\le)0.$$
\vsp
\NI(b) Finally, assume that $\phi>0$ solves 
$$\Dp \phi-(p-1) \phi^{p-2} \phi_t\ge (\le) 0, \;\;\mbox{in $\Om_T$.}$$
For the change of variables described above, set $\psi(z, \om)=\phi(x,t)^{1/\tht}$, where $\tht>0$. Then 
$D_x\phi=\theta \al \psi^{\theta-1}D_z\psi$ and $\phi_t=\theta \beta \psi^{\theta-1}\psi_\om$. Hence,
\ben
\Dp \phi&-&(p-1)\phi^{p-2} \phi_t\\
&=&\al^p \theta^{p-1}\mbox{div}(\psi^{(\theta-1)(p-1)}|D\psi|^{p-2} D\psi)-(p-1)\theta \beta \psi^{\theta(p-2)+\theta-1}\psi_\tau\\
&=&\al^p\theta^{p-1} \psi^{(\theta-1)(p-1)}\left( \Dp \psi+(\theta-1)(p-1)\frac{|D\psi|^p}{\psi} \right)-(p-1)\theta \beta \psi^{\theta(p-1)-1}\psi_\om\\
&=&\al^p\theta^{p-1}\psi^{(\theta-1)(p-1)} \left(\Dp \psi+(\theta-1)(p-1)\frac{|D\psi|^p}{\psi}-\frac{(p-1)\beta}{\al^p \theta^{p-2}}\psi^{p-2}\psi_\om \right),
\een
since $\tht(p-1)-1=(\tht-1)(p-1)+p-2$. Taking $\beta=\al^p \theta^{p-2}$ and $\tht=p/(p-1)$, we get that
$$\Dp \psi+\frac{|D\psi|^p}{\psi}-(p-1) \psi_\om \ge (\le)0.$$
\NI Analogues of some of our results in this work hold for the partial differential equations in (i)-(iii) in part (a) and part (b). However, our primary interest will be Trudinger's equation. $\Box$ 
\end{rem}
\NI Finally, we state a separation of variables result that will be used in constructing sub-solutions and super-solutions in Section 4.
\begin{lem}\label{sec2.15} Let $\Om\subset \IR^n$ be a bounded domain, $T>0$ and $\lam\in \IR$. Suppose that $\eta(t)\in C^1$ with $\eta>0$. 
If $\phi(x)\in usc(lsc)(\Om),$ $\phi>0,$ solves 
$\Dp \phi+\lam \phi^{p-1} \ge (\le) 0,$ in $\Om$,
and $\psi(x,t)=\phi(x) \eta(t)$ then
$$\Pi(\psi):=\Dp \psi-(p-1)\psi^{p-2} \psi_t\ge(\le)-\psi^{p-1}(t)\left(\lam+(p-1)\frac{\eta^{\prime}}{\eta} \right),\;\;\mbox{in $\Om_T$}.$$
In particular, if $\eta(t)=e^{\ell t}$ and $\ell+\lam/(p-1)\le(\ge) 0$ then $\Pi(\psi)\ge (\le) 0$, in $\Om_T.$
\end{lem}
\begin{proof}  We prove the claim when $\Dp \phi+\lam \phi^{p-1}\ge 0$ in $\Om$. Let $\zeta$ be a test function and $(y,s)\in \Om_T$ be such that
$\psi-\zeta$ has a maximum at $(y,s)$, i.e.,
\eqRef{2.17}
\phi(x) \eta(t)-\zeta(x,t)\le \phi(y) \eta(s)-\zeta(y,s).
\ee
Taking $x=y$ in (\ref{sec2.17}), we obtain $\phi(y)(\eta(t)-\eta(s))\le \zeta(y,t)-\zeta(y,s),$
$\forall\;0<t<T$. This yields $\zeta_t(y,s)=\phi(y) \eta^{\prime}(s)$. Next, taking $t=s$, we obtain
$$\phi(x)-\frac{\zeta(x, s)}{\eta(s)}\le \phi(y)-\frac{\zeta(y,s)}{\eta(s)},\;\;\forall x\in \Om.$$
Thus, $\Dp \zeta(y,s)+\lam \psi^{p-1}(y,s)\ge 0$. Applying these observations, 
\ben
\Dp \zeta(y,s)-(p-1)\psi(y,s)^{p-2}\zeta_t(y,s)
&\ge& -\lam \psi^{p-1}(y,s)-(p-1)\psi(y,s)^{p-2}\phi(y)\eta^{\prime}(s)\\
&=&-\psi^{p-1} \left( \lam+(p-1)\frac{\eta^{\prime}(s)}{\eta(s)} \right)
\een
If $\eta(t)=e^{\ell t}$ then $\eta^{\prime}/\eta=\ell$. The claim holds.
\end{proof}
\vsp
\NI For a fixed $z\in \IR^n$ set $r=|x-z|$. If $f(r)=f(x)$ it is well-known that 
\eqRef{sec2.141}
\Dp f(r)=|f^{\prime}(r)|^{p-2} \left( (p-1) f^{\prime\prime}(r)+\frac{n-1}{r} f^{\prime}(r) \right),\;\;1<p<\infty.
\ee
We now record a simple calculation which will be used in Sections 5 and 6. 
\vsp
\begin{rem}\label{sec2.16}
Let $\gm,\;\lam\in \IR$ and $c>0$. Set $\Lm=(p-n)(p-1)^{-1}-\gm.$
Using (\ref{sec2.141}) we calculate, in $r>0$,
\ben
\Dp (\pm c r^\gm)&=&\pm c^{p-1}|\gm r^{\gm-1}|^{p-2}\left\{ (p-1) \gm (\gm-1)r^{\gm-2}+(n-1)\gm r^{\gm-2} \right\} \\
&=&\pm(p-1)c^{p-1}\gm|\gm|^{p-2}r^{(p-2)(\gm-1)+(\gm-2)} \left( (\gm-1)+\frac{n-1}{p-1}\right) \\
&=&\pm(p-1)c^{p-1}\gm|\gm|^{p-2}r^{p(\gm-1)-\gm} \left(-\Lm\right)=\pm(p-1)c^{p-1}|\gm|^{p}r^{p(\gm-1)} \left( \frac{-\Lm}{\gm r^{\gm} }\right).
\een
Thus,
\bea\label{sec2.17}
\Dp (\pm c r^\gm)&+&\lam(p-1) |D(\pm c r^\gm)|^p\nonumber\\
&=&\lam(p-1)c^p|\gm|^p r^{p(\gm-1)}\pm(p-1)c^{p-1}|\gm|^{p}r^{p(\gm-1)} \left( \frac{-\Lm}{\gm r^{\gm} }\right).\nonumber\\
&=&(p-1)c^{p-1} |\gm|^p r^{p(\gm-1)} \left\{c\lam\pm \left(\frac{-\Lm}{\gm r^{\gm}} \right)  \right\}. \;\;\;\Box
\eea
\end{rem}
\vspp
\section{Maximum and Comparison principles}
\vsp
\NI In this section, we prove a maximum principle and some comparison principles for the equation in (\ref{sec1.1}). The maximum principle is stated for a slightly modified version of the equation in (\ref{sec1.1}) and holds without placing any sign restrictions. The comparison principle is proven using the equation in (\ref{sec1.2}) and implies a quotient type comparison principle for positive solutions to (\ref{sec1.1}). As a consequence, this implies uniqueness for positive viscosity solutions to (\ref{sec1.1}). See \cite{BL} for an analogue for a doubly nonlinear parabolic equation involving the infinity-Laplacian. 
\vsp
\begin{lem}\label{sec3.1}{(Weak Maximum Principle)} Let $\Om\subset \IR^n,\;n\ge 2$ be a bounded domain and $T>0$.\\
\NI (i) If $\phi\in usc(\Om_T\cup P_T)$ solves $\Dp \phi- (p-1)|\phi|^{p-2} \phi_t\ge 0,\;\;\mbox{in $\Om_T$},$ then 
$$\sup_{\Om_{T}}\phi\le\sup_{P_T} \phi=\sup_{\Om_T\cup P_T}\phi.$$
\NI (ii) If $\phi\in lsc(\Om_T\cup P_T)$ and $\Dp \phi- (p-1)|\phi|^{p-2}\phi_t\le 0,\;\;\mbox{in $\Om_T$},$ 
then 
$$\inf_{\Om_T} \phi\ge \inf_{P_T}\phi=\inf_{\Om_T\cup P_T}\phi.$$
\end{lem}
\begin{proof} We prove part (i). We note that since $\Om$ is bounded, we choose a $z\in \IR^n\setminus \overline{\Om}$ and an $0<R<\infty$ such that $\Om\subset B_R(z)$. Fix $z$ and set $r=|x-z|$. 
\vsp
\NI Fix $\tau$ close to $T$ with $\tau<T$. We first show that the weak maximum principle holds in $\Om_\tau$ for any $\tau<T$. Set 
\eqRef{sec3.2}
m=\sup_{\Om_\tau} \phi,\;\;\ell=\sup_{P_\tau}\phi,\; \;c=\sup_{\overline{\Om}}\phi(x,\tau),\;\;\dl=m-\ell\;\;\mbox{and $k=\max(\dl,\;c-\ell)$}.
\ee
We argue by contradiction and assume that $\dl>0.$ Since $\Om_\tau$ is an open set there is a point $(y,s)\in \Om_\tau$ such that $\phi(y,s)>\ell+3\dl/4$ and $0<s<\tau$. Define
$$g(t)= \left\{ \begin{array}{ccc} 0, && 0\le t\le s,\\ (t-s)^4/(\tau-s)^{4}, && s\le t\le \tau. \end{array} \right. $$
Select $0<\ve\le \min(\dl/4,1/2).$ Set
\ben
\psi(x,t)=\psi(r,t)=\ell+\frac{\ve}{4} +kg(t)-\frac{\ve r^2}{160R^2},\;\;\forall(x,t)\in \overline{\Om}_{\tau}.
\een
Then $\psi(x,t)\ge \ell+\ve/8,\;\forall(x,t)\in \overline{\Om}_\tau,$ and $\psi(x,\tau)\ge c+\ve/8,\;\;\forall x\in \overline{\Om}$.
Moreover,
\ben
\phi(y,s)-\psi(y, s)\ge \ell+\frac{3\dl}{4}-\ell-\frac{\ve}{4}=\frac{3\dl}{4}-\frac{\ve}{4}
\ge \frac{3\dl}{4}-\frac{\dl}{16}>\frac{\dl}{4}>0.
\een
Since $\phi-\psi\le 0$ on $\p \overline{\Om}_{\tau}$, the function $\phi-\psi$ has a positive maximum at some point $(z,\tht)\in \Om_{\tau}$. Setting $\rho=|y-z|$ and using (\ref{sec2.141}) and (\ref{sec3.2}), we get
\ben
\Dp \psi(z,\tht)=|\psi'|^{p-2} \left( (p-1)\psi''+\frac{n-1}{r}\psi' \right)(z,\tht)
=-\left( \frac{\ve}{80 R^2}\right)^{p-1} \rho^{p-2} (n-p-2).
\een
Since $g^{\prime}(t)\ge 0$, we have
\eqRef{sec3.3}
\Dp \psi(z,\tht)<0\le (p-1)|\phi(z,\tht)|^{p-2}\psi_t(z,\tht).
\ee
We obtain a contradiction and our assertion holds in $\Om_\tau$ for any $\tau<T.$ 
\vsp
\NI If $\sup_{\Om_T}\phi>\sup_{P_T} \phi$ then there is a point $(y,s)\in \Om_T$ (with $0<s<T$) such that 
$\phi(y,s)>\sup_{P_T} \phi$. Select $s<\bar{s}<T$. Then, $\sup_{P_T}\phi<\phi(y,s)\le \sup_{\Om_{\bar{s}}}\phi\le  \sup_{P_{\bar{s}}}\phi\le \sup_{P_T}\phi.$ This is a contradiction and the lemma holds. Part (ii) may be proven similarly.
\end{proof}
\vsp
\begin{rem}\label{sec3.4}  If $\phi\in C(\Om_T\cup P_T)$ then equality holds in the conclusions of Lemma \ref{sec3.1}. \\
Lemma \ref{sec2.12} implies that a version of Lemma \ref{sec3.1} holds for $\Dp \eta+ (p-1)|D\eta|^p- (p-1)\eta_t =0.$ It is also clear from the proof that
Lemma \ref{sec3.1} applies to $\Dp \eta- \eta_t= 0$, see (\ref{sec3.3}). 
$\Box$
\end{rem}
\vsp
\NI Next, we prove a comparison principle for (\ref{sec1.1}) under the condition that the sub-solutions and the super-solutions are positive in $\Om_T$, i.e, we require the positivity of their respective infima on $\Om_T\cup P_T$. 

\begin{thm}\label{sec3.5}{(Comparison principle)}
Suppose that $\Om\subset \IR^n$ is a bounded domain and $T>0$. Let $u\in usc(\Om_T\cup P_T)$ and $v\in lsc(\Om_T\cup P_T)$ satisfy
$$\Dp u-(p-1) u^{p-2} u_t\ge 0,\;\;\mbox{and}\;\;\Dp v- (p-1) v^{p-2} v_t\le 0,\;\;\mbox{in $\Om_T$}.$$
Assume that $\min(\inf_{\Om_T\cup P_T} u, \inf_{\Om_T\cup P_T} v)>0$. If $\sup_{P_T}v<\infty$ and
$u\le v$ on $P_T$, then $u\le v$ in $\Om_T$. 

\NI In particular, if $u$ and $v$ are solutions and $u=v$ on $P_T$ then $u=v$ in $\Om_T$.
\end{thm}
\NI{\bf Proof:} Clearly, $u>0$ in $\Om_T$, and since $u\le v$ on $P_T$, by Lemma \ref{sec3.1}, $u$ is bounded. Define
$\eta(x,t)=\log u(x,t)$ and $\zeta(x,t)=\log v(x,t).$ Then $\eta$ and $\zeta$ are both bounded, in particular, from below. By Lemma \ref{sec2.12},
\ben
\Dp \eta+(p-1)|D\eta|^p- (p-1) \eta_t\ge 0\;\;\;\mbox{and}\;\;\;\Dp \zeta+(p-1)|D\zeta|^p-(p-1)\zeta_t\le 0,\;\mbox{in $\Om_T$,}
\een
with $\eta\le \zeta$ on $P_T$. The conclusion follows by an adaptation of Theorem 8.2 in \cite{CIL}. \qquad$\Box$
\vsp
\NI Thus, Theorem \ref{sec3.5} implies uniqueness of positive solutions of (\ref{sec1.1}) and solutions of (\ref{sec1.2}). We derive now some further consequences. For ease of presentation, we recall the notation $\G(u)=\Dp u+(p-1)|Du|^p-(p-1)u_t$, see (\ref{sec2.3}).

\begin{cor}\label{sec3.6} Let $\eta\in usc(\Om_T\cup P_T)$ and $\zeta \in lsc(\Om_T\cup P_T)$. Suppose that
$$\G(\eta)\ge 0\;\;\;\mbox{and}\;\;\;\G(\zeta)\le 0,\;\mbox{in $\Om_T$.}$$
If $\eta$ and $\zeta$ are bounded in $\Om_T\cup P_T$ then $  \sup_{\Om_T}(\eta-\zeta)\le   \sup_{P_T}(\eta-\zeta).$
\vsp
\NI Moreover, if $\eta$ and $\zeta$ are solutions then $\eta,\;\zeta\in C(\Om_T\cup P_T)$ and 
$\sup_{P_T} |\eta-\zeta|=\sup_{\Om_T}|\eta-\zeta|.$ In particular, if $\eta=\zeta$ on $P_T$ then $\eta=\zeta$ in $\Om_T$. 
\end{cor}
\begin{proof} Let $k=\sup_{P_T}(\eta-\zeta)$ and $\zeta_k=\zeta+k$. Since $\G(\zeta_k)\le 0$ and $\zeta_k\ge \eta$ on $P_T$, Theorem \ref{sec3.5}, implies that $\sup_{\Om_T}(\eta-\zeta_k)\le \sup_{P_t}(\eta-\zeta_k)=0.$ The last claim follows from Remark \ref{sec3.4}.
\end{proof}
\vsp
\NI Theorem \ref{sec3.5} and Corollary \ref{sec3.6} imply a quotient comparison principle. 
\begin{cor}\label{sec3.7}
Let $u$ and $v$ be as in Theorem \ref{sec3.5}, then 
$$\frac{u}{v}\le \sup_{P_T}\frac{u}{v}, \;\;\;\mbox{or}\;\;\;\frac{u-v}{v}\le\sup_{P_T}\frac{u-v}{v},\;\;\mbox{in $\Om_T$} .$$
\end{cor}
\begin{proof} This follows from Corollary \ref{sec3.6} by writing $\eta=\log u$ and $\zeta=\log v$. For solutions, we obtain
$$\sup_{P_T} \frac{u}{v}=\sup_{\Om_T}\frac{u}{v}\;\;\;\mbox{and}\;\;\;\sup_{P_T}\frac{v}{u}=\sup_{\Om_T}\frac{v}{u}.$$
\end{proof}
\vspp
\section{Proofs of Theorems \ref{sec1.4} and \ref{sec1.5}. Initial data: $2\le p<\infty$}
\vsp
\NI Our proof of the existence of solutions to (\ref{sec1.1}) involves constructing sub-solutions and super-solutions for the problem (see (\ref{sec2.10}) and (\ref{sec2.11})) that are arbitrarily close, in a local sense, to the data specified on the parabolic boundary $P_T$. For this purpose, we divided our work into three sections. In this section we take up the construction for the initial data at $t=0$. Our work is valid for $2\le p<\infty$ and any bounded $\Om$. We take up the side conditions i.e, data specified along $\p\Om\times [0,T)$, in Sections 5 and 6. 
The ideas used to construct the sub-solutions and super-solutions are quite similar to those in \cite{BL}.
\vsp
\NI We recall the definition of $h$ from (\ref{sec1.3}), 
$$h(x,t)=\left\{ \begin{array}{lcr} f(x),&\;\forall x\in\Om,\;t=0,\\ g(x,t),&\;\forall(x,t)\in \p\Om\times[0,T). \end{array}\right. $$
Note that $h(x,t)$ is continuous on $P_T$. Set
\eqRef{sec4.0}
m=\inf_{P_T} h\;\;\mbox{and}\;\;M=\sup_{P_T} h.
\ee
Assume that $0<m< M<\infty$. If $m=M$ then $u(x,t)=M$ is the unique solution of (\ref{sec1.1}).
In this section we work directly with the operator
\eqRef{sec4.1}
\Pi(\phi)=\Dp \phi-(p-1)\phi^{p-2} \phi_t,\;\;\mbox{in $\Om_T$.} 
\ee
Also, recall from Lemma \ref{sec2.3} that if $\phi(x)$, $\phi>0$, solves
$\Dp \phi+\lam \phi^{p-1}\ge (\le)0$ and $\psi(x,t)=\phi(x) e^{-\ell t}$ then
\eqRef{sec4.2}
\Pi(\psi)=\Dp \psi-(p-1) \psi^{p-2} \psi_t\ge (\le) \psi^{p-1} \left\{ \ell (p-1)-\lam \right\}.
\ee
Using (\ref{sec2.141}) we also note that if $r=|x-y|$, for some $y\in \IR^n$, then
\eqRef{sec4.3}
\Dp r^2=\s_p2^{p-1} r^{p-2},\;\;\mbox{in $\IR^n$,}
\ee
where $\s_p=(p+n-2).$ In what follows, $\ve>0$ is small so that $m-2\ve>0$.
\vsp
\NI{\bf Part I. Sub-solutions: $t=0$}
\vsp
\NI Let $y\in \overline{\Om}$. Assume that $h(y,0)>m$, other wise take the sub-solution to be $m$ in all of $\Om_T$. We discuss the cases (a) $y\in \Om$ and (b) $y\in \p\Om$ separately.
\vsp
\NI {\bf Case (a):} Let $y\in \Om$. By continuity, there is a $0<\dl\le$dist$(y, \p\Om)$ such that 
$$h(y,0)-\ve \le h(x,0)\le h(y,0)+\ve,\;\;x\in B_\dl (y).$$ Set $r=|x-y|$ and take
\eqRef{sec4.30}
\phi(x)=h(y,0)-2\ve -(h(y,0)-m)\frac{r^2 }{\dl^2},\;\;\forall x\in \overline{B}_\dl(y),
\ee
Clearly, $\phi(y)=h(y,0)-2\ve$ and $\phi|_{\p B_\dl(y)}=m-2\ve$. Moreover, by (\ref{sec4.3}) and taking $r=\dl$, we have 
\eqRef{sec4.301}
\Dp \phi= -\frac{\s_p2^{p-1}r^{p-2}}{\dl^{2(p-1)}}(h(y,0)-m)^{p-1}\ge - \frac{\s_p2^{p-1}}{\dl^p}(h(y,0)-m)^{p-1},\;\;\;\mbox{in $B_\dl(y)$}.
\ee
Now take
$$\lam= \frac{\s_p2^{p-1}}{\dl^p}\left(\frac{h(y,0)-m}{m-2\ve}\right)^{p-1}.$$
Since $m-2\ve\le \phi\le h(y,0)-2\ve$, using (\ref{sec4.301}) we get that
\eqRef{sec4.4}
\Dp \phi+\lam \phi^{p-1}\ge \lam (m-2\ve)^{p-1}- \frac{\s_p2^{p-1}}{\dl^p}(h(y,0)-m)^{p-1}\ge 0,\;\;\mbox{in $B_\dl(y)$}.
\ee
Taking $\psi(x,t)=\phi(x) e^{-\lam t/(p-1)}$ in $B_\dl(y)\times (0,T),$ it is clear by (\ref{sec4.2}) that $\Pi(\psi)\ge 0$. Also, $\forall(x,t)\in \overline{B}_\dl(y)\times[0,T],$
\bea\label{sec4.5}
&&(i)\;\psi(y,0)=h(y,0)-2\ve,\;(ii)\;m-2\ve\le \psi(x,0)\le h(x,0)-\ve,\;\;\mbox{and}\\
&&(iii)\;(m-2\ve)e^{-\ell t}\le \psi(x,t)\le \psi_y(y,t)=(h(y,0)-2\ve)e^{-\ell t}.\nonumber
\eea
\vsp
\NI Call $R=\overline{B}_\dl(y)\times [0, T)$; extend $\psi$ as follows:
\eqRef{sec4.6}
\psi_y(x,t)=\left\{ \begin{array}{lcr} \psi(x,t), & \forall (x,t)\in R,\\ (m-2\ve) e^{-\lam t/(p-1)}, & \forall (x,t)\in \Om_T\setminus R.
\end{array}\right. \ee
A simple calculation shows that $\psi_y$ is a sub-solution in $\Om_T\setminus R$. We only need to check that $\psi$ is a sub-solution in $\p B_\dl(y)\times[0,T).$
\vsp
\NI Let $(z,s)\in \p B_\dl(y)\times (0,T)$ and $(a,q,X)\in \mathcal{P}^+_{\Om_T}\psi_y(z,s)$. Thus, for $(x,t)\rightarrow (z,s)$, where $(x,t)\in \Om_T,$ we have
\eqRef{sec4.7}
\psi_y(x,t)- \psi_y(z,s)\le a(t-s)+\langle q, x-z\rangle+\frac{ \langle X (x-z), x-z\rangle }{2}+o(|t-s|+|x-z|^2),
\ee
Taking $x=z$ in (\ref{sec4.7}) we have $\psi_y(z,t)- \psi_y(z,s)\le a(t-s)$, as $t\rightarrow s$. Recalling (\ref{sec4.5})(iii) and (\ref{sec4.6}), we get 
\eqRef{sec4.8}
a=-\frac{(m-2\ve)\lam e^{-\lam s/(p-1)} } {p-1}<0.
\ee
\vsp
\NI Next, take $t=s$ in (\ref{sec4.7}) to obtain $\psi_y(x,s)- \psi_y(z,s)\le \langle q, x-z\rangle +o(|x-z|) ,$ as $x\rightarrow z$. Recalling (\ref{sec4.5}) and (\ref{sec4.6}),
$\psi_y(x,t)\ge (m-2\ve)e^{-\lam t/(p-1)}$, in $\Om_T$, and $ \psi_y(x,s)- \psi_y(z,s)\ge 0$. Thus,
$\langle q, x-z\rangle +o(|x-z|)\ge 0 ,$ as $x\rightarrow z$. Clearly, $q=0$. Using (\ref{sec2.8}) and (\ref{sec4.8})
$$|q|^{p-2}\mbox{tr}(X)+(p-2)|q|^{p-4}\sum q_iq_j X_{ij}-(p-1) a\psi_y^{p-2}(z,s)>0. $$ 
\vsp
\NI To summarize, for every $y\in \Om$ and $\ve>0$, small, the function $\psi_y\in C(\overline{\Om}_T)$ is such that\\
\NI (i) $0<\psi_y\le h$ on $P_T$, (ii) $\psi_y(y,0)=h(y,0)-2\ve,$ and (iii) $\Pi(\psi)\ge 0,$ in $\Om_T$. See (\ref{sec4.5}) and (\ref{sec4.6}). 
For later reference, set for every $y\in \Om$ and $\ve>0$, small,
\eqRef{sec4.9}
\al_{y,\ve}(x,t)=\log \psi_y(x,t),\;\;\forall(x,t)\in \overline{\Om}_T. \;\;\;\Box
\ee
\vsp
\NI {\bf Case (b):} Let $y\in \p\Om.$ By continuity, there are $\dl>0$ and $\tau>0$ such that 
$$h(y,0)-\ve\le h(x,t)\le h(y,0)+\ve,\;\;\forall (x,t)\in \Om_T\cap (B_{\dl}(y)\times[0,\tau]).$$
We choose $\phi(x)$ and $\lam$ as in Case (a). Recalling (\ref{sec4.2}), 
we select $\ell\ge \lam/(p-1)$, such that
$$(h(y,0)-2\ve)e^{-\ell \tau}\le m-2\ve.$$ 
Defining $\psi(x,t)=\phi(x) e^{-\ell t},$ in $\overline{B}_\dl(y)\times[0,T]$, we see that $\psi$ is a sub-solution in $B_\dl(y)\times (0,T).$ We observe that
$(m-2\ve)e^{-\ell t}\le \psi(x,t)\le \psi_y(y,t)=(h(y,0)-2\ve)e^{-\ell t},\;\forall(x,t)\in \overline{B}_\dl(y)\times[0,T]$.
Set
$$\hat{\psi}_y(x,t)=\left\{\begin{array}{lcr} \psi(x,t), & \forall (x,t)\in \Om_T\cap (B_\dl(y)\times [0,T))\\ (m-2\ve)e^{-\ell t}, & \forall (x,t)\in \Om_T\setminus (B_\dl(y)\times [0,T)).
\end{array}\right. $$
Verifying that $\hat{\psi}_y$ is a sub-solution in $\Om_T$ is similar to Case (a). For any $y\in \p\Om$ and $\ve>0$, small, define
\eqRef{sec4.10}
\hal_{y,\ve}(x,t)=\log\hat{\psi}_y(x,t),\;\;\forall(x,t)\in \Om_T. \;\;\;\;\;\Box
\ee
\vsp
\NI{\bf Part II. Super-solutions: $t=0$}
\vsp
\NI As done in Part I, we discuss the two cases: (a) $y\in \Om$, and (b) $y\in \p\Om$. The treatment here is quite similar to that in Part I. 
We assume that $h(y,0)<M$, otherwise the function $\psi=M$, in $\Om_T$, is a super-solution.
\vsp
\NI {\bf Case (a):} Let $y\in \Om$. Select $0<\dl\le$dist$(y,\p\Om)$ such that 
$$h(y,0)-\ve\le h(x,0)\le h(y,0)+\ve,\;\;\forall x\in B_{\dl}(y).$$
Set $r=|x-y|$ and consider 
\eqRef{sec4.11}
\phi(x)=\phi(r)=h(y,0)+2\ve +(M-h(y,0)) \frac{r^2}{\dl^2},\;\;\forall x\in \overline{B}_{\dl}(y).
\ee
Thus, $\phi(y)=h(y,0)+2\ve$, $\phi|_{\p B_\dl(y)}=M+2\ve$ and $0<h(y,0)+2\ve\le \phi(x)\le M+2\ve$. 
Choose
$$\lam=\frac{\s_p2^{p-1}}{\dl^p} \left( \frac{M-h(y,0)}{h(y,0)+2\ve} \right)^{p-1}.$$
Using the observation made above and (\ref{sec4.3}), we get in $0\le r\le \dl$,
\ben
\Dp \phi-\lam \phi^{p-1}&=&\frac{\s_p 2^{p-1} r^{p-2}}{\dl^{2(p-1)}} (M-h(y,0))^{p-1}-\lam \phi^{p-1}\\
&\le& \frac{\s_p 2^{p-1}}{\dl^p} (M-h(y,0))^{p-1}-\lam (h(y,0)+2\ve)^{p-1}= 0.
\een
Set $\psi(x,t)=\phi(x) e^{\lam t/(p-1)}$, in $\overline{B}_{\dl}(y)\times [0,T)$. By (\ref{sec4.2}), $\psi$ is a super-solution in $B_{\dl}(y)\times [0,T)$.
Also, $\forall(x,t)\in \overline{B}_\dl(y)\times[0,T]$,
\bea\label{sec4.110}
&&(i)\;\; \psi(y,0)=h(y,0)+2\ve,\;\;\;(ii)\;\;\psi(x,0)\ge h(x,0)+\ve,\;\forall x\in \overline{B}_\dl(y),\;\;\mbox{and}\\
&&(iii)\;\;(h(y,0)+2\ve)e^{\lam t/(p-1)}=\psi_y(y,t)\le \psi(x,t)\le (M+2\ve) e^{\lam t/(p-1)}.\nonumber
\eea
Set $R=\overline{B}_\dl(y)\times [0, T).$ Extend $\psi$ by setting
\eqRef{sec4.111}
\psi_y(x,t)=\left\{ \begin{array}{lcr} \psi(x,t), & \forall (x,t)\in \Om_T\cap R,\\ (M+2\ve) e^{\lam t/(p-1)}, & \forall (x,t)\in \Om_T\setminus R.
\end{array}\right. \ee
By (\ref{sec4.2}), $\psi_y$ is super-solution in $\Om_T\setminus R$. We only need to check that $\psi$ is a super-solution in $\p B_\dl(y)\times[0,T).$
\vsp
\NI Let $(z,s)\in \p B_\dl(y)\times (0,T)$ and $(a,q,X)\in \mathcal{P}^-_{\Om_T}\psi_y(z,s)$, i.e., as $(x,t)\rightarrow (z,s)$,
$$\psi_y(x,t)- \psi_y(z,s)\ge a(t-s)+\langle q, x-z\rangle+\frac{ \langle X (x-z), x-z\rangle }{2}+o(|t-s|+|x-z|^2).$$
We use (\ref{sec4.110}) and (\ref{sec4.111}) and arguing as in Case (a) of Part I. Taking $x=z$ we get 
$$a=\frac{(M+2\ve)\lam e^{\lam s/(p-1)} } {p-1}>0.$$
Next, note that (\ref{sec4.110}) and (\ref{sec4.111}) imply $\psi_y(x,s)-\psi(z,s)\le 0,\;\forall x\in \Om$. Taking $t=s$, we get $q=0$. 
Using (\ref{sec2.8})
$ |q|^{p-2}\mbox{tr}(X)+(p-2)|q|^{p-4}\sum q_iq_j X_{ij}-(p-1) a\psi_y^{p-2}(z,s)<0.$
\vsp
\NI Summarizing: (i) $ \psi_y\ge h,\;\mbox{on $P_T$},$ (ii) $\psi_y(y,0)=h(y,0)+2\ve,$ and (iii)  $\Pi(\psi_y)\le 0,$ in $\Om_T$, see (\ref{sec4.110}) and (\ref{sec4.111}).
Set for every $y\in \Om$ and $\ve>0$, small,
\eqRef{sec4.12}
\beta_{y,\ve}(x,t)=\log \psi_y(x,t),\;\;\forall(x,t)\in \overline{\Om}_T. \;\;\;\Box
\ee
\vsp
\NI{\bf Case (b):} Let $y\in \p\Om$. 
By continuity, there are $\dl>0$ and $\tau>0$ such that 
$$h(y,0)-\ve\le h(x,t)\le h(y,0)+\ve,\;\;\forall (x,t)\in \Om_T\cap (B_{\dl}(y)\times[0,\tau]).$$
We choose $\phi(x)$ and $\lam$ as in Case (a). Recalling (\ref{sec4.2}) 
we select $\ell\ge \lam/(p-1)$, such that
$$(h(y,0)+2\ve)e^{\ell \tau}\ge M+2\ve.$$ 
Defining
$$\psi(x,t)=\phi(x) e^{\ell t},\;\;\mbox{in $B_\dl(y)\times(0,T)$},$$
we see that $\psi$ is a super-solution in $B_\dl(y)\times (0,T).$ 
Also, for every $(x,t)\in \overline{B}_\dl(y)\times[0,T]$,
$$(h(y,0)+2\ve)e^{\lam t/(p-1)}=\psi_y(y,t)\le \psi(x,t)\le (M+2\ve) e^{\lam t/(p-1)}.$$
Set
$$\hat{\psi}_y(x,t)=\left\{\begin{array}{lcr} \psi(x,t), & \forall (x,t)\in \overline{\Om}_T\cap (B_\dl(y)\times [0,T])\\ (M+2\ve)e^{\ell t}, & \forall (x,t)\in \overline{\Om}_T\setminus (B_\dl(y)\times [0,T]).
\end{array}\right. $$
Verifying that $\hat{\psi}_y$ is a sub-solution in $\Om_T$ is similar to Case (a). Now set
\eqRef{sec4.13}
\hat{\beta}(x,t)=\log \hat{\psi}_y(x,t),\;\;\forall(x,t)\in \overline{\Om}_T. \;\;\;\Box
\ee
\vspp
\section{Side Boundary: The case $n<p<\infty$ and the proof of Theorem \ref{sec1.4} }
\vsp
\NI We now take up the task of constructing sub-solutions and super-solutions for the side boundary $\p\Om\times(0,T)$. Unlike the case of the initial conditions (see Section 4)
we work with the equation in (\ref{sec1.2}) which we recall here. 
\eqRef{sec5.1}
\Dp \eta+(p-1)|\eta|^p-(p-1)\eta_t=0,\;\;\mbox{in $\Om_T$, and $\eta(x,t)=\log h(x,t),\;\forall(x,t)\in P_T$.}
\ee
See Lemma \ref{sec2.12}, where the change of variable $\eta=\log u$ is discussed. We take $n<p<\infty$ and $\Om$ is any bounded domain. 
\vsp
\NI Also recall, the notation $\G(v):=\Dp v+(p-1)|Dv|^p-(p-1)v_t$, see (\ref{sec2.3}) and also the definitions and discussion following (\ref{sec2.6}) and (\ref{sec2.8}). 
In particular, we mention that in order for
$\G(\eta)\ge (\le)0$ we require that $\forall(a,p,X)\in \XO \eta(x,t) (\mathcal{P}^-_{\Om_T}\eta(x,t))$ and $\forall(x,t)\in \Om_T$,
$$|q|^{p-2}\mbox{tr}(X)+(p-2)|q|^{p-4}q_iq_j X_{ij}+(p-1)|q|^p-(p-1)a\ge (\le)0.$$
\vsp
\NI Let $m$ and $M$ be as in (\ref{sec4.0}). Fix $\ve>0$, small, so that $m-2\ve>0$. Note that if $m=M$ the $u(x,t)=m$ is the unique solution to (\ref{sec1.1}). We continue to assume 
that $0<m<M<\infty$. Recall the notation $D_{\rho,2\tht}(x,t)=B_\rho(x)\times (t-\tht, t+\tht)$, see (\ref{sec2.2}).
\vsp
\NI Let $(y,s)\in P_T$ where $s>0$. There is a $\dl_0>0$ and $\tau_0>0$, depending on $y$ and $s$, such that
\eqRef{sec5.2}
h(y,s)-\ve\le h(x,t)\le h(y,s)+\ve,\;\;\;\forall(x,t)\in \overline{D}_{\dl_0,2\tau_0}(y,s)\cap P_T.
\ee
For any $\gm \in \IR$, we set
\eqRef{sec5.3}
\Lm=\frac{p-n}{p-1}-\gm. \ee
Let $r=|x|$; recalling Remark \ref{sec2.16}, (\ref{sec2.17}) and (\ref{sec5.3}) (take $c>0$) we have in $r>0$,
\bea\label{sec5.4}
\Dp (\pm c r^\gm)+(p-1) |D(\pm  cr^\gm)|^p  =(p-1)c^{p-1} |\gm|^p r^{p(\gm-1)} \left\{c\pm \left(\frac{-\Lm}{\gm r^{\gm}} \right)  \right\},
\eea
\NI Before we move on to the construction of the various functions, we state a lemma that would be used in this section and in Section 6. 
\vsp
\begin{lem}\label{sec5.40}
Let $\Om\subset \IR^n,\;n\ge 2,$ be a bounded domain, $T>0$ and $O\subset \Om_T$ be a sub-domain. Suppose that $\al\in \IR$ and $w:\Om_T\rightarrow \IR$ is such that $w=\al$ in $\Om_T\setminus O$. 
\ben
\mbox{ (i) If $w\in usc(\Om_T)$, $w\ge \al$, in $O$, and $\G(w)\ge 0$, in $O$, then $\G(w)\ge 0$, in $\Om_T$.}\\
\mbox{ (ii) If $w\in lsc(\Om_T)$, $w\le \al$, in $O$, and $\G(w)\le 0$, in $O$, then $\G(w)\le 0$, in $\Om_T$.}
\een
\end{lem}
\begin{proof} It is clear that we need check the claim only at points on $\p O\cap \Om_T$. We prove part (i), the proof part of (ii) is similar. 
\vsp
\NI Let $(y,s)\in \p O$ and $(a,q,X)\in \mathcal{P}^+_{\Om_T}w(y,s)$. Since $w\ge \al$ and $w(y,s)=\al$, we have
\bea\label{sec5.41}
0\le w(x,t)-w(y,s)\le
a(t-s)+\langle q, x-y\rangle&+&\frac{\langle X(x-y),x-y\rangle}{2}\\
&&+o(|t-s|+|x-y|^2),  \nonumber
\eea
as $(x,t)\rightarrow (y,s)$, where $(x,t)\in \Om_T$. Taking $x=y$ in (\ref{sec5.41}) we have that $a(t-s)+o(|t-s|)\ge 0$, as $t\rightarrow s$, implying that $a=0$. Next, taking $t=s$ in (\ref{sec5.41}), we see that $\langle q, x-y\rangle+o(|x-y|)\ge 0$, as $x\rightarrow z$. We obtain $q=0$ and as a result
$$|q|^{p-2}\mbox{tr}(X)+(p-2) |q|^{p-4} q_i q_j X_{ij} +(p-1)|q|^p-(p-1) a=0,$$
proving that $\G(w)\ge 0$ in $\Om_T$. The lemma is proved.
\end{proof}
\vsp
\NI{\bf Part I: Sub-solutions}
\vsp
\NI Recall (\ref{sec5.2}) and fix $y$, $s$, $\dl_0$ and $\tau_0$. We construct a sub-solution in a region $R$ that lies in $\overline{D}_{\dl_0,2\tau_0}(y,s)$ and then extend it to the rest of $\Om_T$ as a sub-solution. 
In what follows, the quantities $0<\tau\le \tau_0$ and $0<\dl\le \dl_0$, positive constants $k,\;\gamma$ and $c$ are such that
\bea\label{sec5.5}
k=\frac{1}{\tau} \log \left ( \frac{h(y,s)-2\ve}{m-2\ve} \right),\;\;\gm=\frac{p-n}{p-1}, \;\;\mbox{and} \;\;\dl=\left( \frac{k \tau}{c}\right)^{1/\gm}.
\eea
Here we take
\eqRef{sec5.6}
c\ge  \frac{ (k \tau)^{\mu} }{\gm^\gm \tau^{\gm/p} }, \;\;\;\mbox{where}\;\;\mu=\frac{\gm+p(1-\gm)}{p}.
\ee
By (\ref{sec5.5}) and (\ref{sec5.6}), it is clear that $\dl\rightarrow 0$ if $\tau\rightarrow 0$. We now fix a value of $\tau$ such that $0<\tau\le \tau_0$ and then a value of $c$ such that  $0<\dl\le \dl_0$. This also fixes the value of $k$. Our choice of $\gm$ shows that $\Lm=0$, see (\ref{sec5.3}). We comment that choosing $0<\gm< (p-n)/(p-1)$ will also work.
\vsp
\NI We now describe the region $R$. Firstly, $R\subset \overline{D}_{\dl,2\tau}(y,s)$ and is the union of two cusp-like regions $R^+$ and $R^-$, where $R^+$ and $R^-$ are as shown in Figure 1.
\begin{figure}[ht]
\includegraphics[scale=.4]{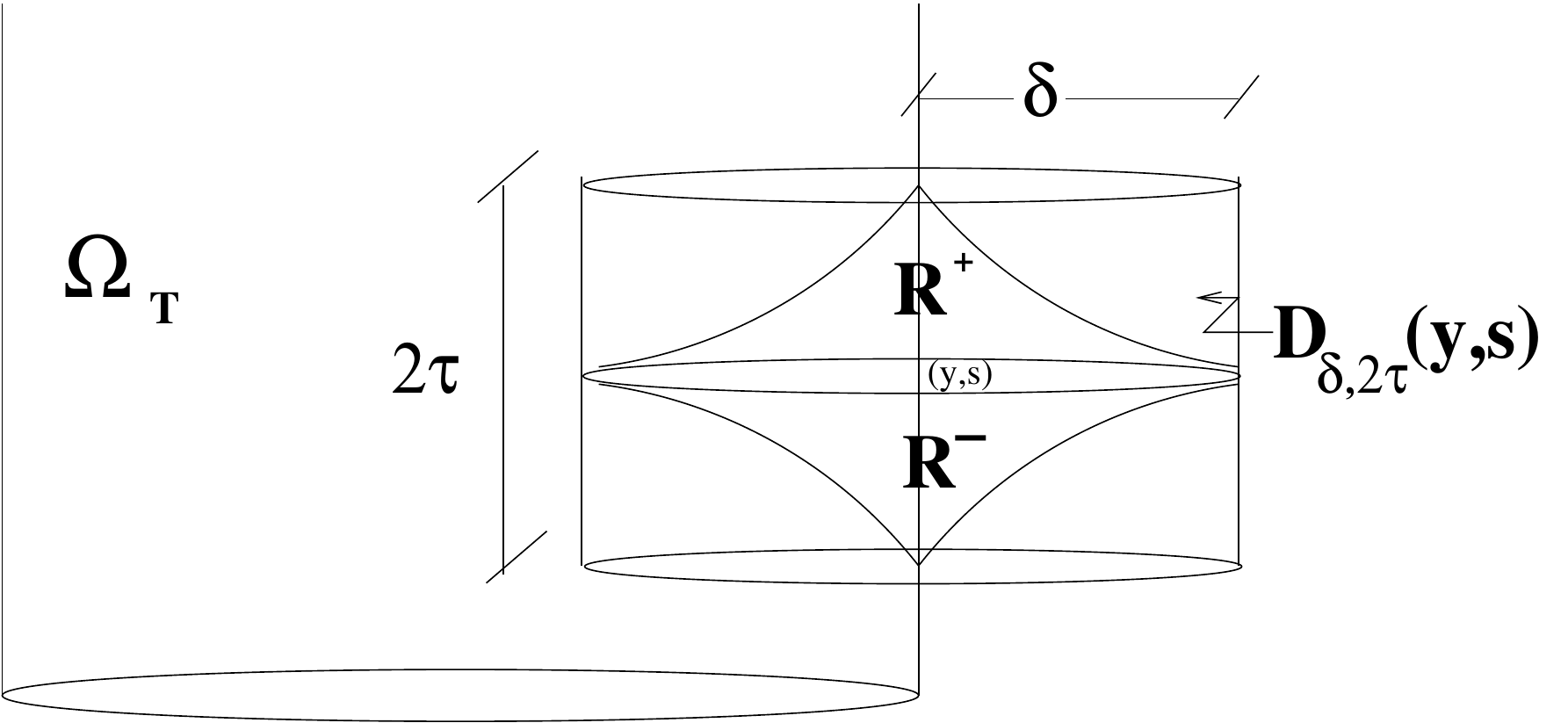}
\caption{Cusps}
\label{fig:cones1}
\end{figure}

 We now describe these more precisely.
Set $r=|x-y|$; define, in $0\le r\le \dl$ and $s-\tau\le t\le s+\tau$, 
\bea\label{sec5.7}
&&R^+\;\mbox{is the cusp:}\;\;k\tau+k(s-t)-cr^\gm\ge 0,\;\;\;s\le t\le s+\tau,\;\;\mbox{and}  \nonumber\\
&& R^- \;\mbox{is the cusp:}\;\;k\tau+k(t-s)-cr^\gamma\ge 0,\;\;\;s-\tau\le t\le s.
\eea
The choice of the various constants in (\ref{sec5.5}) and (\ref{sec5.6}) implies that $\overline{R}$ lies in the cylinder $\overline{B}_{\dl}(y)\times [s-\tau,s+\tau]$. The region $R^-$ is the reflection of $R^+$ about $t=s$. The base $R^+\cap R^-$, common to both the cusp regions, is at $t=s$ and is the spatial ball given by $0\le r \le \dl$. 
\vsp
\NI In $\overline{\Om}_T$, define the {\it{bump function} } 
\eqRef{sec5.8}
\eta(x,t)=\eta(r,t)=\left\{ \begin{array}{lcr} \log(m-2\ve)+k\tau+k(s-t)-cr^\gamma, && \forall(x,t)\in R^+,\\ \log(m-2\ve)+k\tau+k(t-s)-cr^\gamma, && \forall(x,t)\in R^- ,\\
\log(m-2\ve),&& \forall(x,t)\in \overline{\Om}_T\setminus R.
\end{array}\right.
\ee
From (\ref{sec5.2}) and (\ref{sec5.8}), we see that
\ben
&&(i) \; \eta(y,s)=\log(h(y,s)-2\ve),\;\;\;(ii) \; \eta(x,t)\ge \log(m-2\ve), \;\forall(x,t)\in \overline{\Om}_T,\\
&&(iii)\; \log(m-2\ve)\le\eta(x,t)\le\log( h(y,s)-2\ve)\le \log h(x,t),\;\forall(x,t)\in R\cap P_T,
\;\mbox{and}\\
&&(iv)\;\eta(x,t)\le \log h(x,t),\;\forall(x,t)\in P_T.
\een
\NI If we show that $\eta$ is a sub-solution in $\Om_T$, the observations (i)-(iv), listed above, would then imply that $\eta$ is a sub-solution of (\ref{sec1.2}). We first show that $\eta$ is a sub-solution in $R\cap \Om_T$. Theorem \ref{sec5.40} will then show that $\eta$ is a sub-solution in $\Om_T$. We consider: (a) $t\ne s$, and (b) $t=s$. In the following we take $0<r<\dl$, the case $r=\dl$ is contained in Theorem \ref{sec5.40}.
\vsp
\NI{\bf (a) $t\ne s$:} The function $\eta$ is $C^{\infty}$ in the interior of $R\cap \Om_T$ ($t\ne s$), since $r>0$ in $\Om_T$. Noting from (\ref{sec5.3}) that $\Lm=0$ and $\gm<1$, using (\ref{sec5.4}) and (\ref{sec5.8}), we get in $0<r< \dl$,
\bea\label{sec5.10}
\Dp \eta+(p-1)|D\eta|^p-(p-1)\eta_t &=&(p-1) \left(  c^p\gamma^{p}r^{p(\gamma-1)}  \pm k  \right)\nonumber\\
&\ge& (p-1) \left(  \frac{c^p\gamma^{p}}{r^{p(1-\gamma)}} - k  \right)\ge (p-1) \left(  \frac{c^p\gamma^{p}}{\dl^{p(1-\gamma)}} - k  \right)
\eea
We now calculate using (\ref{sec5.5}), (\ref{sec5.6}) and the definition of $\mu$,
\ben
\frac{c^p\gamma^{p}}{\dl^{p(1-\gamma)}}&=& c^p \gm^p \left( \frac{c}{k\tau} \right)^{p(1-\gm)/\gm}=c^{p/\gm}\frac{ \gm^p }{(k\tau)^{p(1-\gm)/\gm}}
\ge   \left(  \frac{ (k \tau)^{\mu} }{\gm^\gm \tau^{\gm/p} } \right)^{p/\gm} \frac{ \gm^p }{(k\tau)^{p(1-\gm)/\gm}}\\
&=& \frac{ (k\tau)^{ \{\gm+p(1-\gm) \}/\gm} }{ \tau (k\tau)^{p(1-\gm)/\gm} }=k.
\een
Hence, (\ref{sec5.10}) implies that $\eta$ is a sub-solution. 
\vsp
\NI{\bf (b) $t=s$:} Let $(z,s)\in R$  and $(a,q,X)\in \mathcal{P}^+_{R}\eta(z,s).$  Then $(x,t)\in R$ and
for $(x,t)\rightarrow (z,s)$, 
\eqRef{sec5.11}
\eta(x,t) \le  \eta(z,s) +a(t-s)+\langle q, x-z\rangle+\frac{\langle X(x-z),x-z\rangle}{2}+o(|t-s|+|x-z|^2).
\ee
\NI We now compute $a,\;q$ and $X$. Since $\eta$ is $C^{\infty}$ in $r>0$, using $t=s$ in (\ref{sec5.11}) we get $q=D\eta(z,s)$ and $X\ge D^2\eta(z,s)$. 
Using (\ref{sec5.8}) and taking $x=z$ in (\ref{sec5.11})
$$
a(t-s)+o(|t-s|)\ge \left\{ \begin{array}{lcr} k(t-s), && t\le s,\\  k(s-t), && t\ge s, \end{array} \right.\;\;\;\;\mbox{as}\;\;t\rightarrow s.
$$
Hence, $-k\le a\le k.$ Thus, from (\ref{sec5.8}) and the discussion following (\ref{sec5.10}), we get 
\ben
|q|^{p-2}\mbox{tr}(X)&+&(p-2)|q|^{p-4} q_iq_j X_{ij}+(p-1)|q|^p-a(p-1)\\
&\ge& \Dp \eta(z,s)+(p-1)|D\eta(z,s)|^p-k(p-1)\ge 0.
\een
Thus $\eta$ is a sub-solution in $R\cap \Om_T$. Define for any $(y,s)\in \p\Om\times(0,T)$ and $\ve>0$, small,
\eqRef{sec5.12}
\nu_{(y,s),\ve}(x,t)=\eta(x,t),\;\;\forall(x,t)\in \overline{\Om}_T. \;\;\;\Box
\ee
\vsp
\NI{\bf Part II: Super-solutions}
\vsp
\NI Let $\tau_0$ and $\dl_0$ be as in (\ref{sec5.2}). For ease of presentation, set $\al=(p-n)/(p-1)$, then $\Lm=\al-\gm$. 
Let $0<\tau\le \tau_0$, $0<\dl\le \dl_0$ and set
\bea\label{sec5.14}
k=\frac{1}{\tau} \log\left(\frac{M+2\ve}{h(y,s)+2\ve}\right),\;\;\gamma=\frac{\al}{1+2k\tau},\;\;\mbox{and}\;\;\dl=\left( \frac{k\tau}{c} \right)^{1/\gm}.
\eea
We take
\eqRef{sec5.15}
c\ge \left(\frac{2 (k\tau)^\mu }{\tau \Lm \gm^{p-1} } \right)^{\gm/p},\;\;\;\mbox{where}\;\;\mu=\frac{p(1-\gm)}{\gm}+2.
\ee
It is clear from (\ref{sec5.14}) and (\ref{sec5.15}) that $\dl\rightarrow 0$ as $\tau\rightarrow 0$. Fix a value $0<\tau\le \tau_0$ and then a value of $c$ such that $\dl\le \dl_0$.  
Also, our choice of $\gm$ implies that 
\eqRef{sec5.16}
\Lm=\al-\gm=\frac{2k\tau \al}{1+2k\tau}\;\;\;\mbox{and}\;\;\;c\gm \dl^\gm=\gm k\tau=\frac{k\tau \al}{1+2k\tau} =\frac{\Lm}{2}.
\ee
\vsp
\NI We construct the super-solution in a region $R$ the union of two cusps $R^+$ and $R^-$. These are defined as follows. 
 \ben
&& R^+\;\mbox{is the cusp-region:} \; k\tau+k(s-t)-cr^\gm\ge 0,\;\;s\le t\le s+\tau,\;\;\mbox{and}\\
&& R^-\;\mbox{is the cusp-region:} \; k\tau+k(t-s)-cr^\gm\ge 0,\;\;s-\tau\le t\le s.
\een
Clearly, $\overline{R}\subset \overline{B}_\dl(y)\times[s-\tau, s+\tau]$.  
Define the {\it indent function}  in $\Om_T$ as follows:
\eqRef{sec5.17}
\eta(x,t)= \left\{ \begin{array}{lcr}  \log(M+2\ve)+cr^\gm-k\tau -k(s-t), && \forall(x,t)\in R^+,\\ \log(M+2\ve)+c r^\gm-k\tau-k(t-s), && \forall(x,t)\in R^-,\\
\log(M+2\ve), && \forall(x,t)\in \overline{\Om}_T\setminus R. \end{array} \right.
\ee
From (\ref{sec5.14}), (\ref{sec5.15}) and (\ref{sec5.17}), we see that
\ben
&&(i) \; \eta(y,s)=\log(h(y,s)+2\ve),\;\;\;(ii)\; \eta(x,t)\le \log(M+2\ve), \;\forall(x,t)\in \overline{\Om}_T,\\
&&(iii)\; \log h(x,t)\le \log( h(y,s)+2\ve)\le \eta(x,t)\le\log(M+2\ve),\;\forall(x,t)\in R\cap P_T,\;\mbox{and}
\\
&&(iv) \;\;\eta(x,t)\ge \log h(x,t),\;\forall(x,t)\in P_T.
\een
\NI If we show that $\eta$ is a super-solution in $\Om_T$, the observations (i)-(v), listed above, would then imply that $\eta$ is a super-solution of (\ref{sec1.2}). We first show that $\eta$ is a super-solution in $R\cap \Om_T$. Theorem \ref{sec5.40} then shows that $\eta$ is a super-solution in $\Om_T$. We consider the cases: (a) $t\ne s$, and (b) $t=s$. 
\vsp
\NI{\bf (a) $t\ne s$:} Note that $\eta_t=\pm k$. Using (\ref{sec5.4}), (\ref{sec5.14}), (\ref{sec5.15}) and (\ref{sec5.16}), we calculate in $0<r\le \dl$,
\bea\label{sec5.18}
\Dp \eta+(p-1)|D\eta|^p-(p-1)\eta_t 
&=&(p-1) \left\{  c^{p-1}\gamma^{p-1}r^{p(\gamma-1)-\gamma} ( c\gamma r^\gamma-\Lm)\pm k  \right\}\nonumber\\
&\le &(p-1) \left\{ c^{p-1}\gamma^{p-1}r^{p(\gamma-1)-\gamma} ( c\gamma \dl^\gm-\Lm)+ k  \right\}\nonumber\\
&\le &(p-1) \left\{k- \frac{ c^{p-1}\gamma^{p-1}\Lm}{ 2\dl^{p(1-\gamma)+\gamma} }  \right\}.
\eea
Recalling (\ref{sec5.15}), (\ref{sec5.16}) and noting that $p-1+ 1+ p(1-\gm)/\gm=p/\gm$, we have 
\ben
\frac{ c^{p-1}\gamma^{p-1}\Lm}{ 2\dl^{p(1-\gamma)+\gamma} } &=&c^{p-1}  c^{1+p(1-\gm)/\gm}   \frac{\gm^{p-1}\Lm  } {2 (k\tau)^ {1+p(1-\gm)/\gm} } =
c^{p/\gm}\frac{\gm^{p-1}\Lm  } {2 (k\tau)^ {1+p(1-\gm)/\gm} }\\
&\ge& \left( \frac{2 (k\tau)^\mu }{\tau \Lm \gm^{p-1} } \right)  \left(   \frac{\gm^{p-1}\Lm  } {2 (k\tau)^ {1+p(1-\gm)/\gm} }   \right)=k.
\een
This together with (\ref{sec5.18}) yields that $\G(\eta)\le 0$.
\vsp
\NI{\bf (b) $t=s$:} We now show that $\eta$ is a super-solution in $R\cap\Om_T$ when $t=s$ and $0<r<\dl$.
Let $(z,s)\in R\cap \Om_T$ and $(a,q,X)\in  \mathcal{P}^-_{\Om_T}{\etb}(z,\tht)$, i.e., as $(x,t)\rightarrow (z,s)$, $(x,t)\in \Om_T$,
\eqRef{sec5.19}
\eta(x,t)-\eta(z,s)\ge a(t-s)+\langle q, x-z\rangle+\frac{ \langle X(x-z), x-z\rangle}{2}+o(|t-s|+|x-y|^2).
\ee
We take $x=z$ in (\ref{sec5.19}) and use (\ref{sec5.17}) to see that
$$a(t-s)+o(|t-s|)\le \left\{ \begin{array}{lcr} k(t-s), && s\le t\le s+\tau,\\ k(s-t), && s-\tau\le t\le s. \end{array} \right.$$
Thus, $-k\le a\le k$. Since $\eta$ is $C^2$ in $r>0$, $q=D\eta(z,s)$ and $X\le D^2\eta(z,s)$. Using the calculations done in (a), see (\ref{sec5.18}), we have
\ben
|q|^{p-2}\mbox{tr}X&+&(p-2)|q|^{p-4}\sum q_iq_j X_{ij}+(p-1)|q|^p-(p-1)a\\
&\le& \Dp\eta+(p-1)|D\eta|^p+(p-1)k\le 0.
\een
This shows that $\eta$ is a super-solution in the interior of $R\cap \Om_T$. For every $(y,s)\in \p\Om\times(0,T)$ and $\ve>0$, small, define
\eqRef{sec5.20}
\hat{\nu}_{(y,s),\ve}=\eta,\;\;\mbox{in $\overline{\Om}_T$}.\;\;\;\Box
\ee
\NI Recall that the functions $\al_{y,\ve},\;\hat{\al}_{y,\ve}$ and $e^{\nu_{(y,s),\ve}}$ are the required sub-solutions for (\ref{sec1.1}), see (\ref{sec4.9}), 
(\ref{sec4.10}) and (\ref{sec5.12}). Next,  $\beta_{y,\ve},\;\hat{\beta}_{y,\ve}$ and $e^{\hat{\nu}_{(y,s),\ve}}$ are the required super-solutions for (\ref{sec1.1}), see (\ref{sec4.12}) and (\ref{sec4.13}). These six functions are in $C(\overline{\Om}_T)$ and the Perron method implies Theorem \ref{sec1.4}.
\vspp
\section{Side Conditions: The case $2\le p\le n$ and proof of Theorem \ref{sec1.5}}
\vsp
\NI In this section, we assume that $\Om$ satisfies a uniform outer ball condition. To be more precise: there is a $\rho_0>0$ such that for each $y\in \p\Om$, if $0<\rho\le \rho_0$ then there is a $z\in \IR^n\setminus \Om$ such that ball $B_{\rho}(z)\subset \IR^n\setminus\Om$ and $y\in \p B_{\rho}(z)\cap \p\Om.$ The center of the region $R$ is the center of the outer ball and lies outside $\overline{\Om}_T$.  Moreover, $R$ lies in a cylindrical shell. Fix $\ve>0$, small, such that $m-2\ve>0$. 
\vsp
\NI Let $(y,s)\in P_T$ where $s>0$. Recall the notation $D_{\rho,2\tht}(x,t)=B_\rho(x)\times (t-\tht, t+\tht)$. 
There is a $\dl_0>0$ and $\tau_0>0$, small, depending on $y$ and $s$, such that
\eqRef{sec6.0}
h(y,s)-\ve\le h(x,t)\le h(y,s)+\ve,\;\;\;\forall(x,t)\in \overline{D}_{\dl_0,2\tau_0}(y,s)\cap P_T.
\ee
For any $\gm>0$, we define (see (\ref{sec5.3}))
\eqRef{sec6.1}
\Lm=\gm-\frac{n-p}{p-1}.
\ee
Recall that $m=\inf_{P_T}h$, $M=\sup_{P_T}h$, and assume that $0<m\le M<\infty$. 
\vsp
\NI{\bf Part I:  Sub-solutions }
\vsp
\NI By our hypothesis, there are $z\in \IR^n\setminus \Om$ and $0<\rho\le \rho_0$ such that $B_\rho(z)\subset \Om^c$ and $y\in \p B_\rho(z)\cap \p\Om$. Note that $z$ depends on $\rho$. Set $r=|x-z|$; the region $R$ will be in the cylindrical shell $(\overline{B}_{\rho+\dl}(z)\setminus B_\rho(z))\times [s-\tau, s+\tau]$, where $\rho$, $\dl$ and $\tau$ will be determined below. To begin with we require that this shell be in $D_{\dl_0, 2\tau_0}(y,s)$ and this is achieved if $\rho+\dl\le \dl_0/2$. We fix a value of $0<\tau\le \tau_0$ and impose that 
\vsp
\NI Set
\eqRef{sec6.2}
\gm>\frac{n-p}{p-1},\;\;k=\frac{1}{\tau} \log\left( \frac{h(y,s)-2\ve}{m-2\ve} \right)\;\;\;\mbox{and}\;\;\;\dl=\rho \left\{\left(\frac{1}{1-\rho^\gm k\tau}\right)^{1/\gm}-1\right\}.
\ee
Choose 
\eqRef{sec6.2r}
0<\rho\le \min\left\{ \rho_0,\;(k\tau)^{-1/\gm},\; \left( \frac{A}{1+k\tau A} \right)^{1/\gm}\right\},\;\;\mbox{and}\;\;A= \left( \frac{\gm^{p}}{k} \right)^{\gm/(p(1+\gm))} 
\ee
where $\rho$ is small enough to ensure that $\rho+\dl\le \dl_0/2$. This is possible since the function $\dl=\dl(\rho)$ is increasing and its range is $[0,\infty)$. 
\vsp
\NI The region $R$ is the union of two regions $R^+$ and $R^-.$ We now describe these more precisely.
For $0<\rho\le r\le \rho+\dl$ and $s-\tau\le t\le s+\tau$, we define the regions
\bea\label{sec6.3}
&&R^+\;\mbox{is the cusp:}\;\;k\tau+k(s-t)+r^{-\gm}-\rho^{-\gm}\ge 0,\;\;s\le t\le s+\tau,\;\;\mbox{and}  \nonumber\\
&& R^- \;\mbox{is the cusp:}\;\;k\tau+k(t-s)+r^{-\gm}-\rho^{-\gm}\ge 0,\;\;s-\tau\le t\le s.
\eea
The region $R^-$ is the reflection of $R^+$ about $t=s$. The base $R^+\cap R^-$, common to both the regions, is at $t=s$ and is the spatial annulus given by $\rho\le r\le \rho+\dl.$ See Figure 2.
\begin{figure}[ht]
\includegraphics[scale=.4]{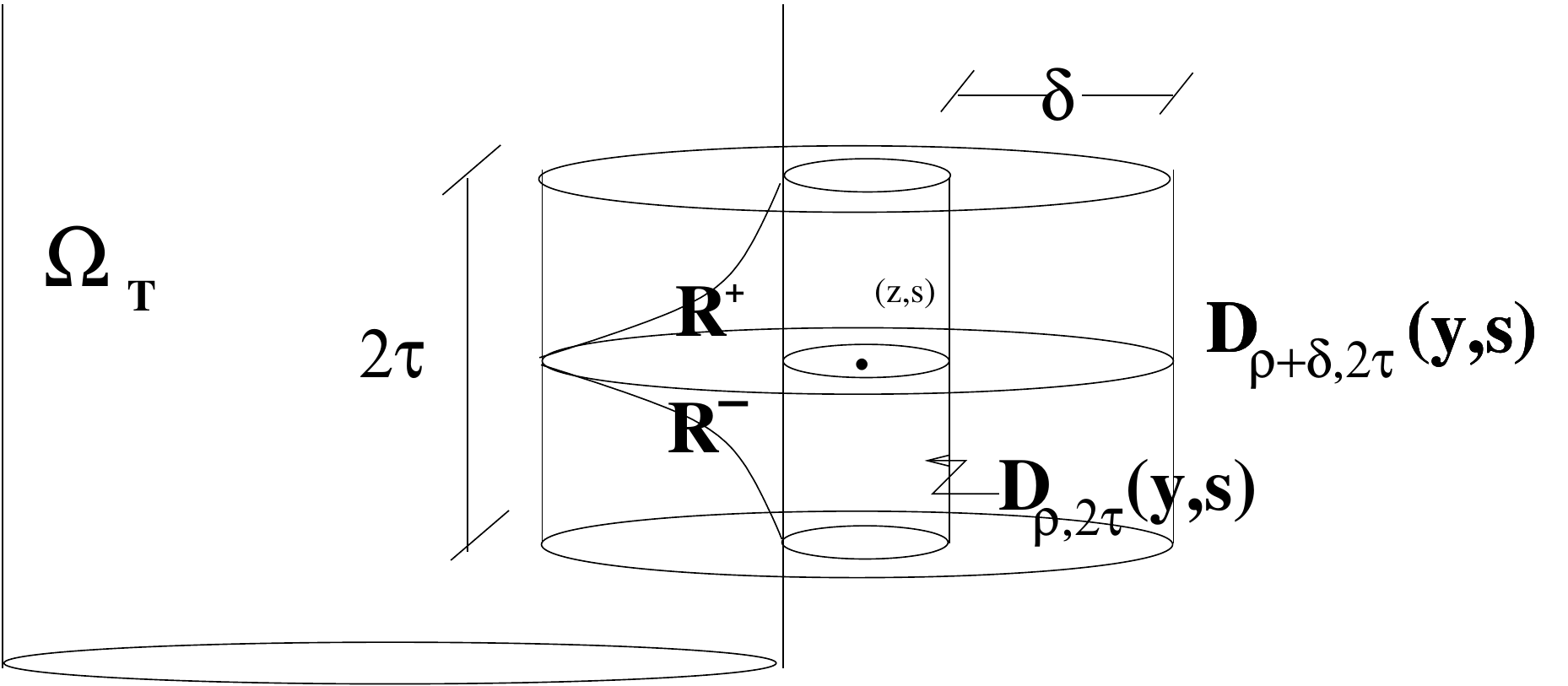}
\caption{Regions $R^+$ and $R^-$}
\label{fig:cones2}
\end{figure}
From (\ref{sec6.2}), $\overline{R}\subset (\overline{B}_{\rho+\dl}(z)\setminus B_\rho(z))\times [s-\tau, s+\tau]$.
\vsp
\NI Define the following {\it{bump function} } in $\Om_T$:
\eqRef{sec6.4}
\eta(x,t)=\eta(r,t)=\left\{ \begin{array}{lcr} \log(m-2\ve)+k\tau+k(s-t)+r^{-\gm}-\rho^{-\gm}, && \forall(x,t)\in R^+,\\ \log(m-2\ve)+k\tau+k(t-s)+r^{-\gm}-\rho^{-\gm}, && \forall(x,t)\in R^- ,\\
\log(m-2\ve)&& \forall(x,t)\in \overline{\Om}_T\setminus R
\end{array}\right.
\ee
From (\ref{sec6.2}), (\ref{sec6.3}) and (\ref{sec6.4}), we see that
\ben
&&(i) \; \eta(y,s)=\log(h(y,s)-2\ve),\;\;\;(ii) \;  \eta(x,t)\ge \log(m-2\ve), \;\forall(x,t)\in \overline{\Om}_T,   \\
&&(iii)\; \log(m-2\ve)\le\eta(x,t)\le\log( h(y,s)-2\ve)\le \log h(x,t),\;\forall(x,t)\in R\cap P_T,\;\mbox{and}\\
&&(iv)\;\eta(x,t)\le \log h(x,t),\;\forall(x,t)\in P_T.
\een
\NI If we show that $\eta$ is a sub-solution in $\Om_T$, the observations (i)-(iv), listed above, would then imply that $\eta$ is a sub-solution of (\ref{sec1.2}). We first show that $\eta$ is a sub-solution in $R\cap \Om_T$. We consider: (a) $t\ne s$, and (b) $t=s$. We show (a), the proof of (b) is similar  to that in Part I of Section 5. Theorem \ref{sec5.40} then shows $\eta$ is a sub-solution in $\Om_T$.
\vsp
\NI{\bf (a) $t\ne s$:} The function $\eta$ is $C^{\infty}$ in the interior of $R\cap \Om_T$, for $t\ne s$. Using (\ref{sec5.4}) and (\ref{sec6.4}), we get in $\rho\le r\le \rho+\dl$, $s-\tau\le t\le s+\tau$, 
\bea\label{sec6.7}
\Dp \eta+(p-1)|D\eta|^p-(p-1)\eta_t&=&(p-1) \left[ \frac{\gm^p}{r^{p(\gm+1)}} \left\{1+ \frac{r^\gm\Lm}{\gm }  \right\}\pm k\right]\nonumber\\
&\ge& (p-1) \left( \frac{\gm^p}{r^{p(\gm+1)}}- k\right)\nonumber\\
&\ge& (p-1) \left( \frac{\gm^p}{(\rho+\dl)^{p(\gm+1)}}- k\right).
\eea
Using (\ref{sec6.2}) and (\ref{sec6.2r}), 
$$\rho+\dl=\left(\frac{\rho^\gm}{1-\rho^\gm k\tau}\right)^{1/\gm}\le A^{1/\gm}=\left(\frac{\gm^p}{k}\right)^{1/(p(1+\gm))}.$$ 
Thus, (\ref{sec6.7}) implies $\Dp \eta+(p-1)|D\eta|^p-(p-1)\eta_t \ge 0$, in $R\cap \Om_T$,
$t\ne s.$ Set for each $(y,s)\in P_T$, 
\eqRef{sec6.70}
\hat{\nu}_{(y,s),\ve}=\eta\;\;\mbox{ on $\overline{\Om}_T$.} \ee
\vspp
\NI{\bf Part II: Super-solutions}
\vsp
\NI Our treatment differs slightly from the one in Section 5 in that we employ scaling and work with an altered equation.
\vsp 
\NI We utilize the change of variables described in part (a) (iii) of Remark \ref{sec2.14}. Let $\lam>0$; set $\om=\lam^{p-2} t$ and 
$\Om_{\lam^{p-2} T}=\{( x, \lam^{p-2} t):\;(x,t)\in \Om_T\}$. Suppose that $w(x, \om)=v(x,t)/\lam$ then the following holds:
$\Dp v+(p-1)|Dv|^{p}-(p-1)v_t\ge (\le) 0$, in $\Om_{T}$, if and only if
\eqRef{sec6.71}
\Dp w+\lam(p-1)|Dw|^{p}-(p-1)w_\om \ge (\le) 0,\;\;\mbox{in $\Om_{\lam^{p-2}T}$.}
\ee
\vsp
\NI Let $(y,s)\in P_T,\;s>0,$ and call
\eqRef{sec6.72}
\al=\log\left( \frac{M+2\ve}{h(y,s)+2\ve}\right).
\ee
We fix a value of $\lam>0$, small, such that
\eqRef{sec6.73}
\al \lam<1.
\ee
\NI Set 
$$\om=\lam^{p-2} t,\;\; \hs=\lam^{p-2}s,\;\; \hT=\lam^{p-2} T\;\; \mbox{and}\;\;\hh(x,\om)=h(x,t)^{1/\lam}.$$
\NI Our goal is to construct a super-solution $\varphi(x,\om)$ of (\ref{sec6.71}), i.e.,
\eqRef{sec6.74}
\Dp \varphi+\lam(p-1)|D\varphi|^{p-2}-(p-1)\varphi_\om\le 0,\;\;\mbox{in $\Om_{\hT},$ \quad $\varphi\ge \log\hh$ on $P_{\hT}$},
\ee
such that $\varphi(y,\hs)$ is close to $\hh(y,\hs)$.
By (\ref{sec6.71}) the function $\eta(x,t)=\lam \varphi(x, \om)$ is then a super-solution of (\ref{sec1.2}) with $\eta(y,s)$ close to $h(y,s)$.
\vsp
\NI Recalling (\ref{sec6.72}) and (\ref{sec6.73}), choose $0<\hve\le \ve$, small, such that
\eqRef{sec6.75}
\hh(y,\hs)+2\hve\le \left( h(y,s)+2\ve \right)^{1/\lam}\;\;\;\mbox{and}\;\;\;\lam \log\left( \frac{M^{1/\lam}+2\hve}{\hh(y,\hs)+2\hve}\right)<1.
\ee
Next, there are $\tau_0>0$ and $\dl_0>0$, small, such that
\eqRef{sec6.18}
\hh(y,\hs)-\hve\le \hh(x,\om)\le \hh(y,\hs)+\hve,\;\;\;\forall(x,\om)\in \overline{D}_{\dl_0,\tau_0}(y,\hs)\cap P_{\hT}.
\ee
\vsp
\NI We fix $0<\tau\le \tau_0$ and set $\hat{M}=M^{1/\lam}$. Define 
\eqRef{sec6.20}
k=\frac{1}{\tau} \log\left(\frac{\hat{M}+2\hve}{\hh(y,\hs)+2\hve}\right).
\ee
Recalling (\ref{sec6.75}) fix values of $0<\tht<1$ and $\gm>(n-p)/(p-1)$ such that
\eqRef{sec6.201}
\lam k\tau=\frac{\tht^2 \Lm}{\gm},
\ee
also see (\ref{sec6.1}).
For $c>0$, to be chosen later, set (use (\ref{sec6.201}))
\eqRef{sec6.21}
\rho^\gm=\frac{c\tht}{k\tau}=\frac{c\lam\gm }{\tht\Lm}\;\;\;\mbox{and}\;\; \dl=\left(\frac{c\tht}{k\tau}\right)^{1/\gm} \left\{ \left( \frac{1}{1-\tht} \right)^{1/\gm}-1\right\}=\rho\left\{ \left( \frac{1}{1-\tht}\right)^{1/\gm}-1\right\}.   \ee
The expression for $\dl$ follows by setting $k\tau=c(\rho^{-\gm}-(\rho+\dl)^{-\gm})$. Note we will select $c$, small, so that $\rho+\dl\le \dl_0/2$. 
\vsp
\NI We now describe the region $R$. By the outer ball condition, for each $0<\rho\le \rho_0$, there is a $z\in \IR^n\setminus \overline{\Om}$ (depending perhaps on $\rho$) such that $B_\rho(z)\subset \IR^n\setminus \overline{\Om}$ and
$y\in \p B_\rho(z)$. Define $r=|x-z|$: set in $\rho\le r\le \rho+\dl$ and $\hs-\tau\le \om\le \hs+\tau,$
\bea\label{sec6.22}
&&R^+\;\mbox{is the region:}\;\;k(\hs-\om)+k\tau+cr^{-\gm}-c\rho^{-\gm}\ge 0,\;\;\hs\le \om\le \hs+\tau,\;\;\mbox{and}  \nonumber\\
&& R^- \;\mbox{is the region:}\;\;k(\om-\hs)+k\tau+cr^{-\gm}-c \rho^{-\gm}\ge 0,\;\;\hs-\tau\le \om\le \hs.
\eea
by (\ref{sec6.21}) the region $R$ lies in the cylindrical shell $(\overline{B}_{\rho+\dl}(z)\setminus B_\rho(z) )\times[\hs-\tau, \hs+\tau]$.
\vsp
\NI In $\Om_T$, define the {\it{indent function} } 
\eqRef{sec6.23}
\varphi(x,\om)=\left\{ \begin{array}{lcr} \log(\hat{M}+2\hve)-k(\hs-\om)-k\tau+c\rho^{-\gm}-c r^{-\gm}, && \forall(x,\om)\in R^+,\\ \log(\hat{M}+2\hve)-k(\om-\hs)-k\tau+c\rho^{-\gm}-c r^{-\gm}, && \forall(x,\om)\in R^- ,\\
\log(\hat{M}+2\hve),&& \forall(x,\om)\in \overline{\Om}_{\hT}\setminus R.
\end{array}\right.
\ee
From (\ref{sec6.20}), (\ref{sec6.21}), (\ref{sec6.22}) and (\ref{sec6.23}) we see that
\ben
&&(i) \; \varphi(y,\hs)=\log(\hh(y,\hs)+2\hve),\;\;\;(ii) \; \varphi(x,t)\le \log(\hat{M}+2\hve), \;\forall(x,\om)\in \overline{\Om}_{\hT},\\
&&(iii)\; \log \hh(x,t)\le \log( \hh(y,\hs)+2\hve)\le \varphi(x,\om)\le\log(\hat{M}+2\hve),\;\forall(x,t)\in R\cap P_{\hT},\;\mbox{and}\\
&&(iv)\;\varphi(x,\om)\ge \log \hh(x,\om),\;\forall(x,\om)\in P_{\hT}.
\een
\NI If we show that $\varphi$ is a super-solution in $\Om_{\hT}$, the observations (i)-(iv), listed above, would then imply that $\eta(x,t)=\lam \varphi(x, \om)$ is a super-solution of (\ref{sec1.2}). To do this we first show that $\eta$ is a super-solution in $R\cap \Om_{\hT}$. We consider: (a) $t\ne s$, and (b) $t=s$. Theorem \ref{sec5.40} will then show that $\varphi$ is a super-solution in $\Om_{\hT}$. Proof of (b) is similar to the proof in Part II of Section 5.
\vsp
\NI For ease of presentation, call
$$\G_\lam(\phi)=\Dp \phi+\lam(p-1)|D\phi|^p-(p-1)\phi_\om.$$

\NI{\bf (a) $t\ne s$:} Applying (\ref{sec5.4}) and using (\ref{sec6.23}) in $\rho\le r\le \rho+\dl$ and 
$\hs-\tau\le t\le \hs+\tau$,
\bea\label{sec6.25}
\G_\lam(\varphi)
&=&\frac{(p-1) c^{p-1}\gm^p}{ r^{p(1+\gm)}} \left(c\lam- \frac{r^\gm\Lm}{\gm}   \right)\pm (p-1)\hk \nonumber\\
&\le&(p-1) \left\{ \frac{c^{p-1}\gm^{p-1}}{ r^{p(1+\gm)-\gm}} \left(   \frac{c\lam \gm}{r^\gm}- \Lm   \right)+k\right\} \nonumber \\
&\le&(p-1) \left\{   \frac{(c\gm)^{p-1}}{ r^{p(1+\gm)-\gm}} \left( \frac{c\lam\gm}{\rho^\gm}- \Lm \right)  + k\right\}.
\eea
Using (\ref{sec6.21}), (\ref{sec6.25}) leads to, in $\rho\le r\le \rho+\dl$ and 
$\hs-\tau\le t\le \hs+\tau$,
\eqRef{sec6.26}
\G_\lam(\varphi)\le (p-1)\left\{  k -\frac{(1-\tht)\Lm(c\gm)^{p-1}}{ r^{p(1+\gm)-\gm}} \right\}\le (p-1)\left\{  k -\frac{(1-\tht)\Lm(c\gm)^{p-1}}{ (\rho+\dl)^{p(1+\gm)-\gm}} \right\}.
\ee
Call $\vartheta=p(1+\gm)-\gm$ and observe that
\eqRef{sec6.260}
1+\frac{\vartheta}{\gm}=\frac{p(1+\gm)}{\gm}\;\;\mbox{and}\;\;p-1-\frac{\vartheta}{\gm}=-\frac{p}{\gm}.
\ee
Recall from (\ref{sec6.21}) that $\rho+\dl=\rho(1-\tht)^{-1/\gm}$. Calculating (see (\ref{sec6.26})) using (\ref{sec6.21}) and (\ref{sec6.260}) we obtain
\ben
\frac{(1-\tht)(c\gm)^{p-1} \Lm}{ (\rho+\dl)^{\vartheta}}&=&\frac{(1-\tht)(1-\tht)^{\vartheta/\gm} (c\gm)^{p-1} \Lm}{\rho^\vartheta}=\Lm\gm^{p-1}(1-\tht)^{p(1+\gm)/\gm} \left(\frac{c^{p-1}}{\rho^\vartheta}\right)\\
&=&  \Lm \gm^{p-1}(1-\tht)^{p(1+\gm)/\gm}\left\{c^{p-1} \left( \frac{ k\tau}{ c\tht}\right)^{\vartheta/\gm} \right\}\\
&=& \Lm \gm^{p-1} (k\tau)^{\vartheta/\gm}\left( \frac{(1-\tht)^{p(1+\gm)}} {\tht^\vartheta} \right)^{1/\gm}\left( \frac{c^{p-1}} {c^{\vartheta/\gm}}\right)
=\frac{L}{c^{p/\gm}},
\een
where $L=L(\Lm, \gm, p, \tht, k\tau)$. Thus, we obtain from (\ref{sec6.26})
$$\G_\lam(\varphi)\le (p-1) \left( k -\frac{L}{c^{p/\gm}} \right)\le 0,$$
if we choose $c$ small enough. Our choice now determines $\rho,\;\dl$ and satisfies $\rho+\dl\le \dl_0/2$. Set for each $\ve>0$ and $(y,s)\in P_T$, $\hat{\mu}_{(y,s),\ve}=\eta$ on $\Om_T$.
\vsp
\NI Recall that the functions $\al_{y,\ve},\;\hat{\al}_{y,\ve}$ and $e^{\hat{\nu}_{(y,s),\ve}}$ are the required sub-solutions for (\ref{sec1.1}), see (\ref{sec4.9}), (\ref{sec4.10}) and (\ref{sec6.70}). Next,  $\beta_{y,\ve},\;\hat{\beta}_{y,\ve}$ and $e^{\hat{\mu}_{(y,s),\ve}}$ are the required super-solutions for (\ref{sec1.1}), see (\ref{sec4.12}) and (\ref{sec4.13}). These are in $C(\overline{\Om}_T)$ and
the Perron method implies Theorem \ref{sec1.5}.

\end{document}